\documentclass[12pt, reqno]{amsart}
\usepackage[utf8]{inputenc}
\usepackage[english]{babel} %Lingua

\usepackage{amssymb}
\usepackage{bbm}
\usepackage{amsthm}  %Teoremi
\usepackage{amsmath} %\text
\usepackage{enumerate} %Liste con numeri romani
\usepackage{mathrsfs}
\usepackage{upgreek} % tau del troncamento
\usepackage{fix-cm} % per cambiare sizefont quando voglio, comando \begin{sizepar}{ grandezza voluta} { + 6} o \size{}{}
\usepackage{adjustbox}
\usepackage{verbatim} % COMMENTS
\usepackage{stmaryrd} % LLBRACKET
\usepackage{dirtytalk} %Per virgolette \say{}
\usepackage[mathscr]{eucal}
\usepackage{microtype}

\usepackage{tikz-cd}

\usepackage{aliascnt}

\usepackage{graphicx}

\usepackage{kantlipsum} % for text filler

\calclayout

\usepackage[inner = 22mm, outer = 22mm, top = 25mm, bottom = 32mm, a4paper]{geometry}

\makeatletter
\def\l@subsection{\@tocline{2}{0pt}{2.5pc}{5pc}{}}
\makeatother

\DeclareRobustCommand{\SkipTocEntry}[5]{}

\usepackage{xcolor} %colorix
\usepackage[colorlinks=true, linkcolor=blue, urlcolor=red, citecolor=brown]{hyperref} %cosÃ¬ decido i colori dei link hyperref
\addto\extrasenglish{

  }

\let\oldfootnotemark\footnotemark
\let\oldfootnotetext\footnotetext
\let\oldfootnote\footnote
\renewcommand\footnote[1]{\addtocounter{footnote}{1}\hypertarget{fnbackref.\arabic{footnote}}{}\addtocounter{footnote}{-1}\oldfootnote{#1\fnbackref}}
\renewcommand\footnotemark{\addtocounter{footnote}{1}\hypertarget{fnbackref.\arabic{footnote}}{}\addtocounter{footnote}{-1}\oldfootnotemark}
\renewcommand\footnotetext[1]{\oldfootnotetext{#1\fnbackref}}
\newcommand{\fnbackref}{\hyperlink{fnbackref.\arabic{footnote}}{\footnotesize$\uparrow$}}

\theoremstyle{definition}
\newtheorem{dfn}{Definition}[subsection]

\newaliascnt{assumption}{dfn}

\aliascntresetthe{assumption}

\newaliascnt{assumptions}{dfn}

\aliascntresetthe{assumptions}

\theoremstyle{remark}

\newaliascnt{rmk}{dfn}
\newtheorem{rmk}[rmk]{Remark}
\aliascntresetthe{rmk}

\newaliascnt{ex}{dfn}
\newtheorem{ex}[ex]{Example}
\aliascntresetthe{ex}

\theoremstyle{plain}
\newaliascnt{thm}{dfn}
\newtheorem{thm}[thm]{Theorem}
\aliascntresetthe{thm}

\newaliascnt{prop}{dfn}
\newtheorem{prop}[prop]{Proposition}
\aliascntresetthe{prop}

\newaliascnt{prop2}{dfn}

\aliascntresetthe{prop2}

\newaliascnt{lem}{dfn}
\newtheorem{lem}[lem]{Lemma}
\aliascntresetthe{lem}

\newaliascnt{cor}{dfn}
\newtheorem{cor}[cor]{Corollary}
\aliascntresetthe{cor}

\DeclareMathOperator{\Hom}{Hom}
\DeclareMathOperator{\End}{End}
\DeclareMathOperator{\thick}{thick}
\DeclareMathOperator{\Aut}{Aut}
\DeclareMathOperator{\im}{im}
\DeclareMathOperator{\cone}{cone}
\DeclareMathOperator{\id}{id}
\DeclareMathOperator{\Tot}{Tot}
\DeclareMathOperator{\Bl}{Bl}

\newcommand{\derived}{\mathrm{D}}
\newcommand{\bounded}{\mathrm{b}}

\newcommand{\qc}{\mathrm{qc}}
\newcommand{\dk}{\derived_{\qc}(\widehat{X}) / \mcr{K}}
\newcommand{\dqc}{\derived_{\qc}(X_{+})}
\newcommand{\inj}{\mathbf{h}\text{-}\mathbf{inj}}

\newcommand{\gen}[1]{\langle #1 \rangle}
\newcommand{\cgen}[1]{\langle #1 \rangle^{\oplus}}
\newcommand{\tgen}[1]{\langle #1 \rangle^{\thick}}

\newcommand{\mb}[1]{\mathbb{#1}}
\newcommand{\ca}[1]{\mathcal{#1}}
\newcommand{\mcr}[1]{\mathscr{#1}}

\newcommand{\BigWedge}{\mathord{\adjustbox{valign=top,totalheight=0.8\baselineskip}{$\bigwedge$}}} %bigwedge more inline friendly

\newenvironment{sizepar}[2]
    {\fontsize{#1}{#2}\selectfont}
    {\par}

\makeatletter
\newcommand{\nodeequation}[1]{%
  \let\label\ltx@label
  \refstepcounter{equation}%
  #1
  \quad
  (\theequation)%
}
\makeatother

\makeatletter
\newcommand{\nodeequationleft}[1]{%
  \let\label\ltx@label
  \refstepcounter{equation}%
  (\theequation)
  \quad
  #1
}
\makeatother

\makeatletter
\newcommand{\nodeequationcenter}[1]{%
  \let\label\ltx@label
  \refstepcounter{equation}%
  (\theequation)
  #1
}
\makeatother

\makeatletter
\def\namedlabel#1#2{\begingroup
    #2%
    \def\@currentlabel{#2}%
    \phantomsection\label{#1}\endgroup
}
\makeatother

\begin{document}

%%
%% The title of the paper goes here.  Edit to your title.
%%

\title{Spherical functors and the flop-flop autoequivalence}

%%
%% Now edit the following to give your name and address:
%% 

\author{Federico Barbacovi}
\address{Department of Mathematics, University College London}
\email{federico.barbacovi.18@ucl.ac.uk}
%\urladdr{www.math.sc.edu/$\sim$howard} % Delete if not wanted.

%%
%% If there is another author uncomment and edit the following.
%%

%\author{Second Author}
%\address{Department of Mathematics, University of South Carolina,
%Columbia, SC 29208}
%\email{second@math.sc.edu}
%\urladdr{www.math.sc.edu/$\sim$second}

%%
%% If there are three of more authors they are added in the obvious
%% way. 
%%

%%%
%%% The following is for the abstract.  The abstract is optional and
%%% if not used just delete, or comment out, the following.
%%%

\begin{abstract}
  Flops are birational transformations which, conjecturally, induce derived equivalences.
  In many cases an equivalence can be produced as pull-push via a resolution of the birational transformation;
  when this happens, we have a non-trivial autoequivalence of either sides of the flop known as the \emph{flop-flop autoequivalence}.

  We prove that such autoequivalence can be realised as the inverse of a spherical twist around a conservative spherical functor in a natural way.
  More precisely, we prove that a natural, conservative spherical functor exists in a more general framework and that the flop-flop autoequivalence fits into this picture.

  We also give an explicit description of the source category of the spherical functor for standard flops (local model and family case) and Mukai flops.
  We conclude with some speculation about Grassmannian flops and the Abuaf flop.
\end{abstract}

%%
%%  LaTeX will not make the title for the paper unless told to do so.
%%  This is done by uncommenting the following.
%%

\maketitle

%%
%% LaTeX can automatically make a table of contents.  This is done by
%% uncommenting the following:
%%

\tableofcontents

%%
%%  To enter text is easy.  Just type it.  A blank line starts a new
%%  paragraph. 
%%

\section{Introduction}
Let $X$ be a scheme and consider $\derived^{\bounded}(X)$ its bounded derived category of coherent sheaves.
One of the gains in passing to the cohomological world is that we acquire new interesting \emph{symmetries}: $\Aut(X) \subsetneq \Aut(\derived^{\bounded}(X))$.
Describing the autoequivalence group of the derived category is a highly non-trivial task, and even the construction of autoequivalences that are not \emph{standard} requires some work, see e.g. \cite{Seidel-Thomas01}, \cite{Huyb-Thomas06}.

It has been clear for some time now that birational transformations play a prominent role in the construction of derived equivalences \cite{Bridgeland-flops-derived-categories}.
Of particular interest to us is the following definition: we say that two smooth, projective varieties $X_{\pm}$ are \emph{K-equivalent} if there exists a roof $X_{-} \xleftarrow{p} \widehat{X} \xrightarrow{q} X_{+}$ of birational maps such that $p^{\ast}\omega_{X_{-}} \simeq q^{\ast}\omega_{X_{+}}$.
Conjecturally, \cite{Bondal-Orlov-SOD-alg-var}, \cite{Kawamata-K-D-equivalence}, K-equivalent varieties should be derived equivalent: $\derived^{\bounded}(X_{-}) \simeq \derived^{\bounded}(X_{+})$.

If we believe in the validity of this conjecture, we get plenty of autoequivalences that go under the name of \emph{flop-flop autoequivalences}.
It often happens that pull-push through $\widehat{X}$ gives the desired equivalence, e.g. standard flops \cite{Bondal-Orlov-SOD-alg-var}.
However, this is not always the case, e.g. Mukai flops \cite{Namikawa-mukai-flops}, \cite{Kawamata-K-D-equivalence}.
Nevertheless, in many cases there exists \emph{some} roof $X_{-} \xleftarrow{p'} \widehat{X}' \xrightarrow{q'} X_{+}$ such that push-pull through $\widehat{X}'$ gives a derived equivalence.
The price we have to pay is that not only we might have to introduce singularities, but $\widehat{X}'$ might also be reducible, e.g. Mukai flops \cite{Namikawa-mukai-flops}.
In this paper we will be interested in a roof (not necessarily a K-equivalence and, in fact, $X_{\pm}$ will not even have to be smooth) such that pull-push via $\widehat{X}$ is an equivalence;
we will denote $\mathrm{FF}_{+} = q_{\ast} p^{\ast} p_{\ast} q^{\ast}$ and $\mathrm{FF}_{-} = p_{\ast}q^{\ast}q_{\ast}p^{\ast}$ the autoequivalences that arise, and we will call them \emph{flop-flop autoequivalences}

Many examples where Bondal--Orlov, Kawamata conjecture holds are known, e.g. \cite{Bondal-Orlov-SOD-alg-var}, \cite{Namikawa-mukai-flops}, \cite{Namikawa-stratified-mukai-flops}, \cite{Halpern-GIT-15}, \cite{Seg-Abuaf-Flop}, and the study of the resulting autoequivalences has produced examples of the \emph{flop-flop = (inverse) twist} phenomenon.
Namely, in many cases, e.g. \cite{Seg-GIT-LG}, \cite{Addington-Donovan-Meachan}, \cite{AL-P-n-func}, the autoequivalences $\mathrm{FF}_{\pm}$ can be described as the composition of inverses of spherical twists around spherical functors (see \cite{Anno-Log-17} of this beautiful piece of theory).
That \emph{any} autoequivalence can be realised as the inverse of the spherical twist around a spherical functor follows from \cite{Seg-autoeq-spherical-twist}, but one might still wonder whether there exists a \emph{natural} spherical functor whose inverse twist is $\mathrm{FF}_{\pm}$.

The first part of this paper is devoted to give a positive answer to this question.
As it turns out, there is no need to restrict to derived categories of coherent sheaves and we will work in the setup of enhanced, compactly generated, triangulated categories.\footnote{We will call such categories just triangulated from now on.}
We will consider categories admitting arbitrary direct sums (which corresponds to considering $\derived_{\qc}(-)$: the unbounded derived categories of sheaves with quasi coherent cohomologies); this passage is of technical nature but it is not just a formality, see \autoref{ex:where-bound-fails}.

Fix $\mcr{B}_{-} \xleftarrow{\alpha_{-}} \mcr{A} \xrightarrow{\alpha_{+}} \mcr{B}_{+}$ a roof of triangulated categories and functors such that $\alpha_{\pm}$ have fully faithful left adjoints $\alpha_{\pm}^L$.
Denote $\mcr{K} = \ker \alpha_{-} \cap \ker \alpha_{+}$ and $\overline{\alpha}_{\pm} \colon \mcr{A} / \mcr{K} \rightarrow \mcr{B}_{\pm}$ the induced functors.
The fundamental assumption that we make is the following: the functors
\[
	\begin{array}{lcr}
		\Phi_{+} = \alpha_{+} \alpha^L_{-} \colon \mcr{B}_{-} \rightarrow \mcr{B}_{+} & \mathrm{and} & \Phi_{-} = \alpha_{-} \alpha^L_{+} \colon \mcr{B}_{+} \rightarrow \mcr{B}_{-}
	\end{array}
\]
are equivalences.
We will call such a framework a \emph{flop-flop diagram} and we will call the autoequivalences $\Phi_{\pm}  \Phi_{\mp}$ the \emph{flop-flop autoequivalences}.

Recall that a functor $\Psi$ between triangulated categories is called conservative if $\ker \Psi = 0$.
We prove (for the relation between spherical functors and four periodic SODs see \cite{Halpern-Shipman16})

\begin{thm}[\autoref{thm:formal-statement-4-sods}, \autoref{thm:4-sods-quotient}, \autoref{thm:spherical-functor-from-sods}]
	\label{thm:introduction-thm-1}
	We have a four periodic SOD
	\[
		\mcr{A}/ \mcr{K} = \gen{\ker \overline{\alpha}_{-}, \mcr{B}_{-}} = \gen{\mcr{B}_{-}, \ker \overline{\alpha}_{+}} = \gen{\ker \overline{\alpha}_{+}, \mcr{B}_{+}} = \gen{\mcr{B}_{+}, \ker \overline{\alpha}_{-}}
	\]
	which is induced by a four periodic SOD of ${}^{\perp} \mcr{K} \subset \mcr{A}$.
	Furthermore, the functors $\overline{\Psi}_{\pm} : \ker \overline{\alpha}_{\pm} \xrightarrow{\overline{\alpha}_{\pm}} \mcr{B}_{\pm}$ are conservative spherical functors and the inverses of the twists around them are given by $\Phi_{\pm}  \Phi_{\mp}$.
\end{thm}

The above theorem can be applied to the case of a roof which induces an equivalence of derived categories.
Then, it tells us that the autoequivalences $\mathrm{FF}_{+}$ and $\mathrm{FF}_{-}$ can be realised as the inverses of the twists around the spherical functors $\overline{\Psi}_{+} \colon \ker \overline{p}_{\ast} \xrightarrow{\overline{q}_{\ast}} \derived_{\qc}(X_{+})$ and $\overline{\Psi}_{-} \colon \ker \overline{q}_{\ast} \xrightarrow{\overline{p}_{\ast}} \derived_{\qc}(X_{-})$, respectively.

Let us now assume that the maps $p$ and $q$ are proper and have finite Tor dimension, then pushforward and pullback along them induce functors between $\derived^{\bounded}(\widehat{X})$ and $\derived^{\bounded}(X_{\pm})$.
Therefore, it is natural to ask whether a theorem similar to the one above holds for $\derived^{\bounded}(\widehat{X}) / \mcr{K}^{\bounded}$, $\mcr{K}^{\bounded} = \mcr{K} \cap \derived^{\bounded}(\widehat{X})$.
The following is our second result;
notice that we are not making any assumption on the singularities of $X_{\pm}$ and $\widehat{X}$.

\begin{thm}[\autoref{thm:4-sods-bounded}]
	\label{thm:introduction-thm-2}
	Take a roof $X_{-} \xleftarrow{p} \widehat{X} \xrightarrow{q} X_{+}$ that induces a flop-flop diagram and assume that $p$ and $q$ are proper and of finite Tor dimension.
	Then, we have a four periodic SOD
	\[	
		\derived^{\bounded}(\widehat{X}) / \mcr{K}^{\bounded} =
		\gen{\ker \overline{p}_{\ast}, \derived^{\bounded}(X_{-})} = 
		\gen{\derived^{\bounded}(X_{+}), \ker \overline{q}_{\ast}} =
		\gen{\ker \overline{q}_{\ast}, \derived^{\bounded}(X_{+})} =
		\gen{\derived^{\bounded}(X_{-}), \ker \overline{p}_{\ast}} .
	\]
	Furthermore, the functors $\overline{\Psi}^{\bounded}_{+} : \ker \overline{p}_{\ast} \xrightarrow{\overline{q}_{\ast}} \derived^{\bounded}(X_{\-})$ and $\overline{\Psi}^{\bounded}_{-} : \ker \overline{q}_{\ast} \xrightarrow{\overline{p}_{\ast}} \derived^{\bounded}(X_{-})$ are conservative spherical functors and the inverses of the twists around them are given by $\mathrm{FF}_{+}$ and $\mathrm{FF}_{-}$, respectively.
\end{thm}

Let us remark that a version of the above theorem was proved in \cite{BodBon15};
Bodzenta and Bondal work in the context of flopping contractions with fibres of dimension at most one and give an explicit description of $\ker \overline{p}_{\ast}$ is terms of the null-category of the flopping contraction.
The paper \emph{ibidem} was the first to consider the quotient category $\derived^{\bounded}(\widehat{X}) / \mcr{K}^{\bounded}$ and the work that led to the current article started with the aim to generalise Bodzenta--Bondal's work to morphisms without any assumption on the dimensions of the fibres.

One of the consequences of assuming that the morphisms have fibres of dimension at most one is that the functor $\derived^{\bounded}(\widehat{X}) / \mcr{K}^{\bounded} \rightarrow \dk$ is fully faithful.
Our third result shows that, if $p$ and $q$ have fibres of dimensions at most one, then \autoref{thm:introduction-thm-2} is a consequence of \autoref{thm:introduction-thm-1} applied to $\derived_{\qc}(-)$.

\begin{thm}[\autoref{thm:relation-to-bounded-case}]
	Assume the same setup as in \autoref{thm:introduction-thm-2} and that $p$ and $q$ have fibres of dimension at most one.
	Then, the four periodic SOD
	\[	
		\derived_{\qc}(\widehat{X}) / \mcr{K} =
		\gen{\ker \overline{p}_{\ast}, \derived_{\qc}(X_{-})} = 
		\gen{\derived_{\qc}(X_{+}), \ker \overline{q}_{\ast}} =
		\gen{\ker \overline{q}_{\ast}, \derived_{\qc}(X_{+})} =
		\gen{\derived_{\qc}(X_{-}), \ker \overline{p}_{\ast}}
	\]
	restricts to give a four periodic SOD of $\derived^{\bounded}(\widehat{X}) / \mcr{K}^{\bounded}$.
\end{thm}

It is important to notice that one can bypass the category $\derived^{\bounded}(\widehat{X}) / \mcr{K}^{\bounded}$ when $X_{+}$ and $X_{-}$ are smooth.
Indeed, passing to compact objects in \autoref{thm:introduction-thm-1} we get a spherical functor $\overline{\Psi}_{\pm} : (\ker \overline{\alpha}_{\mp})^c \rightarrow \derived_{\qc}(X_{\pm})^c$ and under the smoothness assumption we have $\derived_{\qc}(X_{\pm})^c \simeq \derived^{\bounded}(X_{\pm})$.
This approach is better suited for computations because the four periodic SOD of \autoref{thm:introduction-thm-1} is induced by a four periodic SOD of ${}^{\perp} \mcr{K}$, and thus we can come back to $\derived_{\qc}(\widehat{X})$ to compute morphisms between objects.
It is not clear whether this is possible for $\derived^{\bounded}(\widehat{X}) / \mcr{K}^{\bounded}$ because we do not know whether $\mcr{K}^{\bounded}$ is left admissible in $\derived^{\bounded}(\widehat{X})$, see \autoref{rmk:where-boundedness-fails-II}.

\vspace{0.2cm}

After proving the above theorems, we move on to describe the category $\ker \overline{p}_{\ast} \subset \dk$ in some examples.
Describing the source category of a conservative spherical functor is interesting for at least two reasons;
first, because it can help us to understand the twist around the functor (e.g. SODs of the source category induce, under some assumptions, splittings of the twist, see \cite{Halpern-Shipman16});
second, because even though we know that any autoequivalence is a spherical twist, we do not always have nice descriptions of the spherical functor that realises the autoequivalence as a twist.

We will consider the examples of standard flops (both the local case and in families) and Mukai flops.
In all these examples a splitting of the flop-flop autoequivalence as the composition of inverses of spherical twists is known \cite{Addington-Donovan-Meachan} and thus we can guess what the source category $\ker \overline{p}_{\ast}$ should look like.
More precisely, in \cite{Barb-Spherical-twists} we proved that given spherical functors $\Psi_i : \mcr{C}_i \rightarrow \mcr{D}$, $i=1, \dots, n$, there exists a category $\mcr{C}$ and a spherical functor $\Psi : \mcr{C} \rightarrow \mcr{D}$ such that $T_{\Psi} \simeq T_{\Psi_n}  \dots T_{\Psi_1}$; furthermore, $\mcr{C} = \gen{\mcr{C}_n, \dots, \mcr{C}_1}$.
Hence, we expect the category $\ker \overline{p}_{\ast}$ to have a SOD reflecting the splitting of $\mathrm{FF}_{+}$ in all the examples that we study.
We will prove this expectation true as the following theorems show.

When we consider the local model for standard flops we have $X_{+} = \Tot( \ca{O}_{\mb{P}^n}(-1)^{\oplus n+1})$;
\cite{Addington-Donovan-Meachan} shows that in this case the structure sheaf $\ca{O}_{\mb{P}^n}$ is a spherical object \cite{Seidel-Thomas01} and that there is an isomorphism of functors $\mathrm{FF}_{+} \simeq T_{\ca{O}_{\mb{P}^n(-1)}}^{-1}  \dots  T_{\ca{O}_{\mb{P}^n(-n)}}^{-1}$.
We show\footnote{See e.g. \cite[Def. 2.1.3]{Barb-Spherical-twists} for the definition of right gluing functor.}

\begin{thm}[\autoref{thm:C=D(R)}, \autoref{cor:exceptional-coll-std-flops}]
	For the spherical functor arising from the local model for standard flops, \autoref{section:std-flops}, the category $\ker \overline{p}_{\ast}$ has a full, exceptional collection:
	\[
		\ker \overline{p}_{\ast} = \gen{\derived(k), \dots, \derived(k)}
	\]
	and the right gluing bimodule between the $j$-th and $i$-th component (counting right to left) is $\mathrm{RHom}_{X_{+}}(\ca{O}_{\mb{P}}(-j), \ca{O}_{\mb{P}}(-i))$.
\end{thm}

We prove a similar theorem for standard flops in families \autoref{thm:std-flop-family-version} but we refrain to give a statement here because it would require the introduction of too much notation.

We then move on to consider the example of Mukai flops.
This is the first instance where the roof $\widehat{X}$ inducing the equivalence is not smooth and therefore deserves careful consideration.
The variety we flop is given by $\Tot(\Omega^1_{\mb{P}^n})$ and in this case the structure sheaf $\ca{O}_{\mb{P}^n}$ is a $\mb{P}^n$-object \cite{Huyb-Thomas06}.
\cite{Addington-Donovan-Meachan} shows that $\mathrm{FF}_{+} \simeq P_{\ca{O}_{\mb{P}^n}(-1)}^{-1}  \dots  P_{\ca{O}_{\mb{P}^n}(-n)}^{-1}$ where $P_{\ca{O}_{\mb{P}^n}(-i)}$ is the $\mb{P}$-twist around $\ca{O}_{\mb{P}^n}(-i)$ as defined in \cite{Huyb-Thomas06}.
\cite[ 4]{Seg-autoeq-spherical-twist} gives two constructions to realise a $\mb{P}$-twist as a spherical twist and thus we have two possible choices for the source category of our spherical functor: either $\derived(k[q])$ with $\deg(q) = 2$, or $\derived(k[\varepsilon]/\varepsilon^2)$ with $\deg(\varepsilon) = -1$.
If we try to use the same generator we used for standard flops (which is what one would try at first), we fail because of the singularities of $\widehat{X}$: the object we take is a generator but it is not compact, see \autoref{section:Mukai-flops}.
We can however turn our attention to another generator and prove

\begin{thm}[\autoref{thm:C=D(R)-Mukai-complete}]
	For the spherical functor arising from Mukai flops, \autoref{section:Mukai-flops}, we have
	\[
		\begin{array}{lcr}
			\ker \overline{p}_{\ast} = \gen{\derived(k[\varepsilon]/\varepsilon^2), \dots, \derived(k[\varepsilon]/\varepsilon^2)} & & \deg(\varepsilon) = -1
		\end{array}
	\]
	and the right gluing bimodule between the $j$-th and $i$-th component (counting right to left) is $\mathrm{RHom}_{X_{+}}(\{ \ca{O}_{\mb{P}}(-j)[-1] \rightarrow \ca{O}_{\mb{P}}(-j)[1] \}, \{ \ca{O}_{\mb{P}}(-i)[-1] \rightarrow \ca{O}_{\mb{P}}(-i)[1] \})$ where $\{ \ca{O}_{\mb{P}}(-j)[-1] \rightarrow \ca{O}_{\mb{P}}(-j)[1] \}$ is the canonical degree $2$ extension.
\end{thm}

Using the above theorem we can prove the following statement, which tells us that the only problem in the first generator we considered was indeed the lack of compactness.

\begin{thm}[\autoref{thm:end-algebra-structure-sheaves-mukai}]
	For the spherical functor arising from Mukai flops, \autoref{section:Mukai-flops}, there exists a thick subcategory $\mcr{C} \subset \ker \overline{p}_{\ast}$ such that
	\[
		\begin{array}{lcr}
			\mcr{C} = \gen{\derived(k[q])^c, \dots, \derived(k[q])^c} & & \deg(q) = 2
		\end{array}
	\]
	and the right gluing bimodule between the $j$-th and $i$-th component (counting right to left) is $\mathrm{RHom}_{X_{+}}(\ca{O}_{\mb{P}}(-j), \ca{O}_{\mb{P}}(-i))$.
\end{thm}

We conclude with a couple of examples where we can describe $\mathrm{FF}_{+}$ but an explicit description of $\ker \overline{p}_{\ast}$ eludes our understanding.

\addtocontents{toc}{\SkipTocEntry}
\subsection*{Acknowledgments}

I would like to thank my advisor Ed Segal for many helpful conversations.
I would also like to thank Will Donovan for reading a draft of this paper and providing many suggestions, and Agnieszka Bodzenta for explaining to me some technicalities of \cite{BodBon15}.
This project has received funding from the European Research Council (ERC) under the European Union Horizon 2020 research and innovation programme (grant agreement No.725010).

%\section{Generalities}
%\input{Work-in-progress/Generalities}

%\section{Flop-flop = \texorpdfstring{$\text{twist}^{-1}$}{twist-1}}
%\input{Work-in-progress/flop-flop=inverse-twist-new}

\addtocontents{toc}{\SkipTocEntry}
\subsection*{Notations and Conventions}
In section\autoref{sect:flop-flop-inverse} $k$ will denote a fixed field; in section\autoref{sect:examples} $k$ will be an algebraically closed and of characteristic zero.

We denote $k$-linear triangulated categories as $\mcr{A}$, $\mcr{B}$, $\mcr{T}$, etc., and we denote $\Hom_{\#}^{\bullet}(-,-) = \oplus_{n \in \mb{Z}} \Hom_{\#}(-,-[n])[-n]$, $\End^{\bullet}_{\#}(-) = \Hom^{\bullet}_{\#}(-,-)$.
For a subcategory $\mcr{S} \subset \mcr{T}$ the right orthogonal is defined as $\mcr{S}^{\perp} = \{ E \in \mcr{T} : \Hom^{\bullet}_{\mcr{T}}(S,E) = 0 \; \forall S \in \mcr{S} \}$, and similarly for the left orthogonal. 
A triangulated category admitting arbitrary direct sums will be called \emph{cocomplete} and a functor commuting with arbitrary direct sums will be called \emph{continuous}.
A subcategory $\mcr{S}$ of a cocomplete category $\mcr{T}$ will be called a \emph{cocomplete subcategory} if it is abstractly cocomplete and the inclusion $\mcr{S} \subset \mcr{T}$ is continuous.
Notice that not all subcategories which are abstractly cocomplete are cocomplete subcategories according to this definition, see \cite[Remark 2.1.6]{Barb-Spherical-twists}.
An object $E \in \mcr{T}$ is called compact if the functor $\Hom_{\mcr{T}}(E,-)$ is continuous.
The subcategory of compact objects is denoted $\mcr{T}^c$.
A set of objects $\{E_i\} \subset \mcr{T}$ is called a set of compact generators if $E_i \in \mcr{T}^c$ for every $i$ and if, given $F \in \mcr{T}$, $\Hom^{\bullet}_{\mcr{T}}(E_i, F) = 0$ for every $i$ implies $F \simeq 0$.
For simplicity of exposition we will assume that our triangulated categories are compactly generated by a single compact object but all the arguments work if we have a set of compact generators.
Given a set of objects $\{E_i \} \subset \mcr{T}$ we will denote $\cgen{\{E_i\}}$ (resp. $\tgen{\{E_i\}}$) the smallest cocomplete (resp. thick, \emph{i.e.}, closed under taking direct summands) subcategory containing $\{E_i\}$.

We always assume our triangulated categories to be enhanced (in whichever framework one prefers) and we assume functors between them to be enhanced.

When we work with spherical functors we use the definition of \cite{Anno-Log-17}.
For a spherical functor $\Psi \colon \mcr{C} \rightarrow \mcr{D}$ we denote $\Psi^L$ and $\Psi^R$ its left and right adjoint, respectively, and we set $T_{\Psi} = \cone(\Psi \Psi^R \rightarrow \id)$.

Given an abelian category $\ca{A}$, objects $E_i \in \ca{A}$, and maps $f_i \colon E_i \rightarrow E_{i+1}$ such that $f_{i+1} \circ f_i = 0$, we will denote the complex associated to this data as $\left\{ \dots \xrightarrow{f_{i-1}} E_i \xrightarrow{f_{i}} E_{i+1} \xrightarrow{f_{i+1}} \dots \right\}$ with $E_i$ in degree $i$.
If the number of $E_i$'s appearing is finite, the rightmost term appearing is placed in degree $0$.

When we consider functors between derived categories, we assume them to be implicitly derived.

\section{Flop-flop = (inverse) twist}
\label{sect:flop-flop-inverse}
\subsection{General case}
\label{sect:general-case-result}
Let
\begin{equation}
	\label{eqn:std-diagram}
	\mcr{B}_{-} \xleftarrow{\alpha_{-}} \mcr{A} \xrightarrow{\alpha_{+}} \mcr{B}_{+}
\end{equation}
be a diagram of cocomplete triangulated categories and continuous functors.
Assume that $\alpha_{\pm}$ have fully faithful left adjoints $\alpha_{\pm}^L$ and that the functors
\begin{equation}
	\begin{array}{lcr}
		\Phi_{+} = \alpha_{+} \alpha^L_{-} \colon \mcr{B}_{-} \rightarrow \mcr{B}_{+} & \mathrm{and} & \Phi_{-} = \alpha_{-} \alpha^L_{+} \colon \mcr{B}_{+} \rightarrow \mcr{B}_{-}
	\end{array}
\end{equation}
are equivalences.
In the following, we will call the autoequivalences $\Phi_{\pm}  \Phi_{\mp}$ arising from such a configuration the flop-flop autoequivalences.

\begin{rmk}
	The first place where a setup similar to the above appeared is \cite{BodBon15}.
	There the authors consider $f_{-} : X_{-} \rightarrow Y$ a morphism satisfying certain assumptions (among which there is $(f_{-})_{\ast}\ca{O}_{X_{-}} \simeq \ca{O}_Y$ and that the fibres of $f_{-}$ must have at most dimension $1$), take $f_{+} : X_{+} \rightarrow Y$ a flop of $f_{-}$, and set $\mcr{B}_{\pm} = \derived^{\bounded}(X_{\pm})$, $\mcr{A} = \derived^{\bounded}(X_{-} \times_Y X_{+})$, and $\alpha_{\pm} = (f_{\pm})_{\ast}$.
	In \emph{ibidem} the authors also prove the statement analogous to \autoref{thm:4-sods-quotient} for their setup.
\end{rmk}

The fully faithfulness assumption implies that the inclusions $\ker \alpha_{\pm} \hookrightarrow \mcr{A}$ have left adjoints given by $\pi_{\pm} = \cone(\alpha^L_{\pm} \alpha_{\pm} \rightarrow \id)$.
Denote $\mcr{S}_{\pm} := \cgen{ \im (\pi_{\mp}  \alpha^{L}_{\pm})}$ and $\mcr{K} := \ker \alpha_{-} \cap \ker \alpha_{+}$ where $\im(-)$ and $\ker -$ denote the essential image and the kernel, respectively.

In the following, when we embed $\mcr{B}_{\pm}$ in $\mcr{A}$ we write $\mcr{B}_{\pm} \subset \mcr{A}$ in place of $\alpha^L_{\pm} \mcr{B}_{\pm} \subset \mcr{A}$.
For the definition of SOD and right/left admissibility see \cite{BonKap89}, \cite{Bondal-Orlov-SOD-alg-var}.

\begin{thm}
	\label{thm:formal-statement-4-sods}
	Let $\mcr{B}_{-} \xleftarrow{\alpha_{-}} \mcr{A} \xrightarrow{\alpha_{+}} \mcr{B}_{+}$ be a diagram as in \eqref{eqn:std-diagram}.
	Then, $\mcr{K}$ is left admissible in $\mcr{A}$ and we have a four periodic SOD
	\begin{equation}
		\label{eqn:4-sod-k}
		{}^{\perp}\mcr{K} = \gen{\mcr{S}_{+}, \mcr{B}_{-}} = \gen{\mcr{B}_{-}, \mcr{S}_{-}} = \gen{\mcr{S}_{-}, \mcr{B}_{+}} = \gen{\mcr{B}_{+}, \mcr{S}_{-}}.
	\end{equation}
\end{thm}

\begin{proof}
	It follows from \autoref{prop:SODs} and \autoref{prop:SODs-reversed} below.
\end{proof}

Before beginning the proof, let us state a few consequences of the above theorem.
Let us denote $\mcr{A}/\mcr{K}$ the quotient category, $\overline{\alpha}_{\pm} : \mcr{A} / \mcr{K} \rightarrow \mcr{B}_{\pm}$ the induced functors, and similarly $\overline{\alpha}^{L}_{\pm} = Q  \alpha^{L}_{\pm}$ where $Q : \mcr{A} \rightarrow \mcr{A} / \mcr{K}$ is the quotient functor.

\begin{thm}
	\label{thm:4-sods-quotient}
	The category $\mcr{A} / \mcr{K}$ has a four periodic SOD
	\begin{equation}
		\label{eqn:4-sod-quotient}
		\mcr{A} / \mcr{K} = \gen{\ker \overline{\alpha}_{-}, \mcr{B}_{-}} = \gen{\mcr{B}_{-}, \ker \overline{\alpha}_{+}} = \gen{\ker \overline{\alpha}_{+}, \mcr{B}_{+}} = \gen{\mcr{B}_{+}, \ker \overline{\alpha}_{-}}.
	\end{equation}
	Futhermore, the quotient functor induces equivalences $Q: \mcr{S}_{\mp} \rightarrow \ker \overline{\alpha}_{\pm}$.
\end{thm}

\begin{proof}
	Notice that we have SODs $\mcr{A} = \gen{\ker \alpha_{\pm}, \mcr{B}_{\pm}}$ because of the fully faithfulness of $\alpha_{\pm}^L$.
	Hence, \autoref{thm:formal-statement-4-sods} tells us $\ker \alpha_{\pm} = \gen{\mcr{K}, \mcr{S}_{\mp}}$, and the quotient functor induces an equivalence as claimed.
	Then, the four periodic SOD follows from the four periodic SOD of \autoref{thm:formal-statement-4-sods}.
\end{proof}

\begin{prop}
	\label{prop:formal-statement-4-sods-compact}
	The four periodic SODs \eqref{eqn:4-sod-k} and \eqref{eqn:4-sod-quotient} induce four periodic SODs of the subcategory of compact objects.
\end{prop}

\begin{proof}
	This is a completely general statement: if we have a four periodic SOD
	\begin{equation}
		\label{eqn:4-sod-general}
		\mcr{C} = \gen{\mcr{C}_1, \mcr{C}_2} = \gen{\mcr{C}_2, \mcr{C}_3} = \gen{\mcr{C}_3, \mcr{C}_4} = \gen{\mcr{C}_4, \mcr{C}_1}
	\end{equation}
	in terms of cocomplete subcategories, then it induces a four periodic SOD of compact objects, \emph{i.e.},
	\[
		\mcr{C}^c = \gen{\mcr{C}_1^c, \mcr{C}_2^c} = \gen{\mcr{C}_2^c, \mcr{C}_3^c} = \gen{\mcr{C}_3^c, \mcr{C}_4^c} = \gen{\mcr{C}_4^c, \mcr{C}_1^c}.
	\]
	Indeed, the fourth SOD in \eqref{eqn:4-sod-general} tells us that the inclusion $i_{1} \colon \mcr{C}_1 \hookrightarrow \mcr{C}$ has a right adjoint $i_1^R$, and therefore the first SOD has a right gluing functor\footnote{See e.g. \cite[Def. 2.1.3]{Barb-Spherical-twists}.} $i_1^R i_2$.
	As we are assuming that the inclusions $i_j : \mcr{C}_j \hookrightarrow \mcr{C}$ are continuous, the distinguished triangle $i_{1}i_{1}^R \rightarrow \id \rightarrow i_{4}i_4^{L}$ tells us that $i_1^R$ is continuous.
	Hence, $i_1^R i_2$ is continuous, and therefore we get $\mcr{C}^c = \gen{\mcr{C}_1^c, \mcr{C}_2^c}$ by \cite[Lemma 2.1.9]{Barb-Flop-flop}.
	The other SODs can be proven similarly.
\end{proof}

\begin{thm}
	\label{thm:spherical-functor-from-sods}
	The functors $\Psi_{\pm} = \alpha_{\pm}  i_{\mcr{S}_{\pm}} \colon \mcr{S}_{\pm} \rightarrow \mcr{B}_{\pm}$ are conservative spherical functors and the inverses of the twists around them are given by $\Phi_{\pm}  \Phi_{\mp}$.
	The same statement holds for the functors they induce on the quotient $\overline{\Psi}_{\pm} \colon \ker \overline{\alpha}_{\mp} \rightarrow \mcr{B}_{\pm}$.
	Furthermore, the functors $\Psi'_{\pm} = \alpha_{\pm}  i_{\ker \alpha_{\pm}} \colon \ker \alpha_{\pm} \rightarrow \mcr{B}_{\pm}$ are spherical and are given by the gluing of $\Psi_{\pm}$ with the zero functor.
\end{thm}

\begin{proof}
	The first two statements follow from \autoref{thm:formal-statement-4-sods}, \autoref{prop:formal-statement-4-sods-compact}, and \cite[Proposition B.3]{BodBon15}.
	The third follows from the SODs $\ker \alpha_{\pm} = \gen{\mcr{K}, \mcr{S}_{\mp}}$ and \cite[Theorem 3.0.1]{Barb-Spherical-twists}.
\end{proof}

\begin{rmk}
	Notice that the left adjoints to $\Psi_{\pm}$ (resp. $\overline{\Psi}_{\pm}$) are given by $\Psi^L_{\pm} = \pi_{\mp}  \alpha^{L}_{\pm}$ (resp. $\overline{\Psi}^L_{\pm} = Q  \pi_{\mp}  \alpha^{L}_{\pm}$).
\end{rmk}

We now begin the proof of \autoref{thm:formal-statement-4-sods}.

\begin{lem}
	\label{lem:compact-generation-s}
	For any compact generator $G \in \mcr{B}_{\pm}$ the object $\pi_{\mp}\alpha^L_{\pm}G \in \mcr{S}_{\pm}$ is a compact generator.
\end{lem}

\begin{proof}
	We prove the statement for $\mcr{S}_{+}$, the other being analogous.
	Set $G' := \pi_{-}\alpha^{L}_{+}G$; then $G'$ is compact because $\alpha^L_{+}$ and $\pi_{-}$ are left adjoints, and we clearly have $\cgen{G'} \subset \mcr{S}_{+}$.
	By \cite[\href{https://stacks.math.columbia.edu/tag/09SN}{Tag 09SN}]{stacks-project}, as $G$ is a compact generator, every object $E \in \mcr{B}_{+}$ can be written as the homotopy colimit of a system $\left\{ E_{n} \right\}$ where each $E_{n}$ is the iterated cone of maps between direct sums of shifts of $G$.
	As $\pi_{-} \alpha^L_{+}$ commutes with homotopy colimits and $\mcr{S}_{+}$ is cocomplete, we get $\pi_{-}\alpha^L_{+}E \in \mcr{S}_{+}$.
	This implies $\mcr{S}_{+} \subset \cgen{G'}$ and the claim follows.
\end{proof}

\begin{lem}
	\label{lem:right-adjoint-inclusion}
	The inclusions $i_{\mcr{S}_{\pm}} : \mcr{S}_{\pm} \hookrightarrow \mcr{A}$ have continuous right adjoints $i_{\mcr{S}_{\pm}}^R$.
\end{lem}

\begin{proof}
	We prove the statement for $\mcr{S}_{+}$.
	The existence of the right adjoint follows from Brown representability \cite{Neeman-Brown-Rep} using compact generation \autoref{lem:compact-generation-s}.
	With the same notation as in the proof of the above lemma, as we have $\mcr{S}_{+} = \cgen{\pi_{-}\alpha^L_{+}G}$, by \cite[\href{https://stacks.math.columbia.edu/tag/09SR}{Tag 09SR}]{stacks-project} we get $\mcr{S}_{+}^{c} = \tgen{\pi_{-}\alpha^L_{+}G}$.
	Now notice that $\pi_{-}\alpha^{L}_{+}G$ is compact in $\mcr{A}$ because $\Phi_{-}$ is an equivalence and $\alpha_{\pm}^L$ preserve compactness being left adjoints.
	Thus, $i_{\mcr{S}_{+}}$ preserves compactness;
	hence, its right adjoint is continuous.
\end{proof}

\begin{prop}
	\label{prop:SODs}
	The category $\mcr{K}$ is left admissible and we have SODs
	\[
		\mcr{A} = \gen{\mcr{K}, \mcr{S}_{+}, \mcr{B}_{-}} = \gen{\mcr{K}, \mcr{S}_{-}, \mcr{B}_{+}}.
	\]
\end{prop}

\begin{proof}
	The existence of either of the SODs implies that $\mcr{K}$ is left admissible.
	We prove the existence of the first SOD.
	As $\mcr{S}_{+}$ and $\mcr{B}_{-}$ are right admissible and $\mcr{S}_{+} \subset \mcr{B}_{-}^{\perp}$, we have $\mcr{A} = \gen{\mcr{S}_{+}^{\perp} \cap  \mcr{B}_{-}^{\perp}, \mcr{S}_{+}, \mcr{B}_{-}}$.
	Hence, it is enough to prove $\mcr{S}_{+}^{\perp} \cap  \mcr{B}_{-}^{\perp} = \mcr{K}$.

	It is clear that $\mcr{K} \subset \mcr{S}_{+}^{\perp} \cap  \mcr{B}_{-}^{\perp}$.
	Conversely, assume $E \in \mcr{S}_{+}^{\perp} \cap  \mcr{B}_{-}^{\perp}$.
	We have $\alpha_{-}E \simeq 0$;
	moreover, as $\mcr{S}_{+}^{\perp} \cap  \mcr{B}_{-}^{\perp} \subset \ker \alpha_{-}$ and $\pi_{-}$ is the left adjoint of $\ker \alpha_{-} \hookrightarrow \mcr{A}$, we have
	\[
		0 \simeq \Hom^{\bullet}_{\mcr{A}}(\pi_{-}\alpha^L_{+}G, E) \simeq \Hom^{\bullet}_{\mcr{B}_{+}}(G, \alpha_{+}E),
	\]
	where $G \in \mcr{B}_{+}$ is a compact generator.
	This implies $\alpha_{+}E \simeq 0$, and therefore $E \in \mcr{K}$.
\end{proof}

\begin{lem}
	\label{lem:orthogonality-reversed}
	We have $\mcr{B}_{\pm} \subset \mcr{S}_{\pm}^{\perp}$.
\end{lem}

\begin{proof}
	We prove $\mcr{B}_{+} \subset \mcr{S}_{+}^{\perp}$, the other inclusion being similar.
	Take $E,F \in \mcr{B}_{+}$; then
	\[
		\Hom^{\bullet}_{\mcr{A}}(\pi_{-}\alpha_{+}^L E, \alpha^L_{+} F) \simeq
		\cone \left( \Hom^{\bullet}_{\mcr{B}_{+}}(E, F) \rightarrow \Hom^{\bullet}_{\mcr{B}_{-}}(\Phi_{-} E, \Phi_{-}F) \right)[-1] \simeq 0
	\]
	where we used the fully faithfulness of $\alpha_{+}^L$ and that $\Phi_{-}$ is an equivalence by assumption.
\end{proof}

\begin{prop}
	\label{prop:SODs-reversed}
	We have SODs
	\[
		\mcr{A} = \gen{\mcr{K}, \mcr{B}_{-}, \mcr{S}_{-}} = \gen{\mcr{K}, \mcr{B}_{+}, \mcr{S}_{+}}.
	\]
\end{prop}

\begin{proof}
	We show the existence of the first SOD.
	By \autoref{lem:orthogonality-reversed} and the right admissibility of $\mcr{B}_{-}$ and $\mcr{S}_{-}$, we have $\mcr{A} = \gen{\mcr{B}_{-}^{\perp} \cap \mcr{S}_{-}^{\perp}, \mcr{B}_{-}, \mcr{S}_{-}}$.
	Hence, it is enough to prove $\mcr{B}_{-}^{\perp} \cap \mcr{S}_{-}^{\perp} = \mcr{K}$.

	We clearly have $\mcr{K} \subset \mcr{B}_{-}^{\perp} \cap \mcr{S}_{-}^{\perp}$.
	Conversely, assume $E \in \mcr{B}_{-}^{\perp} \cap \mcr{S}_{-}^{\perp}$.
	We have $\alpha_{-}E \simeq 0$;
	moreover, for $G \in \mcr{B}_{-}$ a compact generator we have
	\[
		0 \simeq \Hom^{\bullet}_{\mcr{A}}(\pi_{+} \alpha_{-}^L G, E) \simeq \cone \left( 0 \rightarrow \Hom^{\bullet}_{\mcr{B}_{+}}(\Phi_{+}G, \alpha_{+}E) \right)[-1] \simeq \Hom^{\bullet}_{\mcr{B}_{+}}(\Phi_{+}G, \alpha_{+}E)[-1].
	\]
	As $\Phi_{+}$ is an equivalence $\Phi_{+}G$ is a compact generator, and therefore $\alpha_{+}E \simeq 0$.
	Thus, $E \in \mcr{K}$ and the claim follows.
\end{proof}

\subsection{Bounded derived categories}
\label{subsect:bounded-derived-categories}

A typical example of diagram \eqref{eqn:std-diagram} is the following: take $X_{-}$, $X_{+}$, and $\widehat{X}$ three separated, finite type schemes of finite Krull dimension together with finite type maps $X_{-} \xleftarrow{p} \widehat{X} \xrightarrow{q} X_{+}$ and set $\mcr{B}_{\pm} = \derived_{\qc}(X_{\pm})$, $\alpha_{-} = p_{\ast}$, and $\alpha_{+} = q_{\ast}$.
The assumptions of \autoref{sect:general-case-result} are then equivalent to:
\begin{equation}
	\label{eqn:ass-pushforward}
	\begin{array}{lcr}
		p_{\ast} \ca{O}_{\widehat{X}} \simeq \ca{O}_{X_{-}} & & q_{\ast} \ca{O}_{\widehat{X}} \simeq \ca{O}_{X_{+}}
	\end{array}
\end{equation}
and
\begin{equation}
	\label{eqn:ass-equivalences}
	\begin{array}{lcr}
		q_{\ast}p^{\ast} \colon \derived_{\qc}(X_{-}) \xrightarrow{\simeq} \derived_{\qc}(X_{+}) & & p_{\ast}q^{\ast} \colon \derived_{\qc}(X_{+}) \xrightarrow{\simeq} \derived_{\qc}(X_{-}).
	\end{array}
\end{equation}

Let us now assume that $p$ and $q$ are proper and of finite Tor dimension.
Then $p_{\ast}$, $p^{\ast}$, $q_{\ast}$, and $q^{\ast}$ preserve boundedness and coherence, and therefore we wonder whether four periodic SODs similar to \eqref{eqn:4-sod-k} and \eqref{eqn:4-sod-quotient} exist for bounded derived categories.

Let us denote $\derived^{\bounded}(-)$ the bounded derived category of coherent sheaves.
As $p$ and $q$ are proper and of finite Tor dimension, the right adjoint $p^{\times} : \derived_{\qc}(X_{-}) \rightarrow \derived_{\qc}(X_{+})$ (resp. $q^{\times}$) to $p_{\ast}$ (resp. $q_{\ast}$) preserves boundedness and coherence, see \cite[Lemma 3.12]{Neeman-Grothend-Hochschild}.
Moreover, \cite[Remark 6.1.1]{Neeman-Grothendieck-duality} together with \eqref{eqn:ass-pushforward} imply that $p^{\times}$ and $q^{\times}$ are fully faithful.

From now on, all the functors we write are between bounded derived categories.
Let us denote $\mcr{K}^{\bounded} = \mcr{K} \cap \derived^{\bounded}(\widehat{X})$, $\overline{p}_{\ast} \colon \derived^{\bounded}(\widehat{X}) /\mcr{K}^{\bounded} \rightarrow \derived^{\bounded}(X_{-})$ the functor induced on the quotient, $\overline{p}^{\ast} = Q  p^{\ast}$, $\overline{p}^{\times} = Q  p^{\times}$ and similarly for $q$.
By \cite[Lemma 5.7]{BodBon15} the fully faithfulness of $p^{\ast}$ implies the fully faithfulness of $\overline{p}_{\ast}$ and the adjunction $p^{\ast} \dashv p_{\ast}$ induces an adjunction $\overline{p}^{\ast} \dashv \overline{p}_{\ast}$.
The following lemma shows that the same holds true for $\overline{p}^{\times}$ and the adjunction $\overline{p}_{\ast} \dashv \overline{p}^{\times}$.

\begin{lem}
	\label{lem:morphism-quotient-when-right-orthogonal}
	Let $\mcr{T}$ be a triangulated category and $\mcr{S} \subset \mcr{T}$ be a thick subcategory.
	Let $E \in \mcr{S}^{\perp}$ and denote $Q: \mcr{T} \rightarrow \mcr{T} \left/ \mcr{S} \right.$ the quotient functor.
	Then for any $F \in \mcr{T}$ we have
	\[
		\Hom_{\mcr{T} \left/ \mcr{S} \right.}(Q(F), Q(E)) \simeq \Hom_{\mcr{T}}(F,E)
	\]
\end{lem}

\begin{proof}
	This can be proven similarly to \cite[Lemma 5.7]{BodBon15}.
\end{proof}

Therefore, we have the following SODs of $\derived^{\bounded}(\widehat{X}) /\mcr{K}^{\bounded}$:
\begin{equation}
	\label{eqn:4-sods-bounded}
	\gen{\ker \overline{p}_{\ast}, \overline{p}^{\ast} \derived^{\bounded}(X_{-})} = 
	\gen{ \overline{q}^{\times}\derived^{\bounded}(X_{+}), \ker \overline{q}_{\ast}} =
	\gen{\ker \overline{q}_{\ast}, \overline{q}^{\ast} \derived^{\bounded}(X_{+})} =
	\gen{\overline{p}^{\times} \derived^{\bounded}(X_{-}), \ker \overline{p}_{\ast}}.
\end{equation}

\begin{thm}
	\label{thm:4-sods-bounded}
	Let $X_{-} \xleftarrow{p} \widehat{X} \xrightarrow{q} X_{+}$ as above.
	Then, we have a four periodic SOD
	\[	
		\derived^{\bounded}(\widehat{X}) /\mcr{K}^{\bounded} =
		\gen{\ker \overline{p}_{\ast}, \derived^{\bounded}(X_{-})} = 
		\gen{\derived^{\bounded}(X_{+}), \ker \overline{q}_{\ast}} =
		\gen{\ker \overline{q}_{\ast}, \derived^{\bounded}(X_{+})} =
		\gen{\derived^{\bounded}(X_{-}), \ker \overline{p}_{\ast}} .
	\]
	where $\derived^{\bounded}(X_{-})$ and $\derived^{\bounded}(X_{+})$ are embedded via $\overline{p}^{\ast}$ and $\overline{q}^{\ast}$, respectively.
\end{thm}

\begin{proof}
	We will prove that $\overline{p}^{\ast} \derived^{\bounded}(X_{-}) = \overline{q}^{\times} \derived^{\bounded}(X_{+})$ and $\overline{q}^{\ast} \derived^{\bounded}(X_{+}) =  \overline{p}^{\times} \derived^{\bounded}(X_{-})$ as subcategories of $\derived^{\bounded}(\widehat{X}) /\mcr{K}^{\bounded}$.
	Then, the statement will follow from \eqref{eqn:4-sods-bounded}.

	We prove the first equality, the second being analogous.
	By \autoref{thm:4-sods-quotient} we know that $\overline{p}^{\ast} \derived_{\qc}(X_{-}) = \overline{q}^{\times}\derived_{\qc}(X_{+})$ as subcategories of $\dk$.
	Hence, for every $E \in \derived^{\bounded}(X_{-})$ there exists $F \in \derived_{\qc}(X_{+})$ and an isomorphism $\overline{p}^{\ast} E \simeq \overline{q}^{\times}F$ in $\dk$.
	Applying $\overline{q}_{\ast}$, as $\overline{q}^{\times}$ is fully faithful, we get $F \simeq \overline{q}_{\ast}\overline{p}^{\ast}E \in \derived^{\bounded}(X_{+})$.
	By \cite[Lemma 5.7]{BodBon15} the isomorphism $\overline{p}^{\ast} E \simeq \overline{q}^{\times}F$ comes from a morphism $\beta : p^{\ast} E \rightarrow q^{\times} F$ such that $\cone(\beta) \in \mcr{K}$.
	However, $p^{\ast}E$ and $q^{\times}F$ are bounded complexes with coherent cohomologies, thus $\cone(\beta)$ is too.
	This means that $Q(\beta)$ is an isomorphism in $\derived^{\bounded}(\widehat{X}) /\mcr{K}^{\bounded}$.
	Hence, we have $\overline{p}^{\ast} \derived^{\bounded}(X_{-}) \subset \overline{q}^{\times} \derived^{\bounded}(X_{+})$, and similarly one proves the other containment.
\end{proof}

\begin{cor}
	Let $X_{-} \xleftarrow{p} \widehat{X} \xrightarrow{q} X_{+}$ as above.
	Then, $\Psi_{-} = \overline{p}_{\ast} \vert_{\ker \overline{q}_{\ast}} \colon \ker \overline{q}_{\ast} \rightarrow \derived^{\bounded}(X_{-})$ and $\Psi_{+} = \overline{q}_{\ast} \vert_{\ker \overline{p}_{\ast}} \colon \ker \overline{p}_{\ast} \rightarrow \derived^{\bounded}(X_{+})$ are conservative spherical functors whose inverse twists are given by $p_{\ast} q^{\ast} q_{\ast} p^{\ast}$ and $q_{\ast}p^{\ast}p_{\ast}q^{\ast}$, respectively.
\end{cor}

\begin{rmk}
	If $X_{-}$ and $X_{+}$ are smooth, then $\derived^{\bounded}(X_{\pm}) = \derived_{\qc}(X_{\pm})^c$.
	Therefore, regardless of putting further assumptions on $p$ and $q$, we can always realise the flop-flop autoequivalence of $\derived^{\bounded}(X_{\pm})$ associated to the equivalences \eqref{eqn:ass-equivalences} by restricting the spherical functors of \autoref{thm:spherical-functor-from-sods} to compact objects.
	This approach, when pursuable, is better suited for computations for two reasons.
	First, because it is easier to compute morphisms in $\derived_{\qc}(\widehat{X})$ rather than in $\derived^{\bounded}(\widehat{X}) /\mcr{K}^{\bounded}$.
	Second, because we have generators for $\mcr{S}_{\pm}$, whereas we do not have them in general for $\ker \overline{p}_{\ast} \subset \derived^{\bounded}(\widehat{X}) /\mcr{K}^{\bounded}$, see \autoref{ex:where-bound-fails}. 
\end{rmk}

\subsection{Fibres of dimension at most one}
\label{subsect:rel-dim-one}

We now want to compare our work to \cite{BodBon15}.
Hence, we keep using the setup of \autoref{subsect:bounded-derived-categories} and we further assume that the dimension of the fibres of $p$ and $q$ is bounded above by $1$.

The following lemma is well-known.

\begin{lem}[$\text{\cite[Lemma 3.1]{Bridgeland-flops-derived-categories}}$]
	An object $K \in \derived_{\qc}(\widehat{X})$ is in $\mcr{K}$ if and only if its cohomology sheaves are.
\end{lem}

\begin{proof}
	Our assumptions on $p$ and $q$ imply, by \cite[\href{https://stacks.math.columbia.edu/tag/08D5}{Tag 08D5}]{stacks-project}, that to compute $\ca{H}^i(p_{\ast}E)$ and $\ca{H}^i(q_{\ast}E)$ we can assume $E \in \derived^{+}_{\qc}(\widehat{X})$.
	Then, we have a convergent, second page spectral sequence $\ca{H}^i(p_{\ast}\ca{H}^j(E)) \implies \ca{H}^{i+j}(p_{\ast}E)$ that degenerates at page $2$ and the statement follows.
\end{proof}

Using the above lemma, with a proof similar to the one of \cite[Lemma 5.5]{BodBon15}, we get that $\derived^{\bounded}(\widehat{X}) /\mcr{K}^{\bounded}$ is a full subcategory of $\dk$.
We now prove the following theorem which says that \autoref{thm:4-sods-quotient} and \autoref{thm:4-sods-bounded} produce the same SODs when $p$ and $q$ have fibres of dimension at most $1$.

\begin{thm}
	\label{thm:relation-to-bounded-case}
	Assume that $p$ and $q$ have fibres of dimension at most $1$.
	Then, the SODs of \autoref{thm:4-sods-quotient} restrict to SODs of $\derived^{\bounded}(\widehat{X}) /\mcr{K}^{\bounded}$.
\end{thm}

\begin{proof}
	It is enough to prove that the right adjoint $i^{R}$ to the inclusion $i \colon \ker \overline{p}_{\ast} \hookrightarrow \dk$ preserves boundedness, and similarly for the right adjoint to $\ker \overline{q}_{\ast} \hookrightarrow \dk$.

	Let $E = Q(F) \in \derived^{\bounded}(\widehat{X}) /\mcr{K}^{\bounded}$.
	Then, by \autoref{lem:morphism-quotient-when-right-orthogonal} the exact triangle $i i^R E \rightarrow E \rightarrow \overline{p}^{\times} \overline{p}_{\ast} E$ associated to the SOD $\derived_{\qc}(\widehat{X})/\mcr{K} = \gen{\overline{p}^{\times} \derived_{\qc}(X_{-}), \ker \overline{p}_{\ast}}$ is the image via $Q$ of the triangle $F' \rightarrow F \rightarrow p^{\times} p_{\ast} F$.
	As $F, p^{\times} p_{\ast} F \in \derived^{\bounded}(\widehat{X})$, we have $i i^R E \in \derived^{\bounded}(\widehat{X}) /\mcr{K}^{\bounded}$, and the claim follows.
	The stament for $\ker \overline{q}_{\ast}$ is proved similarly.
\end{proof}

\begin{rmk}
	\label{rmk:Bod-Bond}
	In \cite{BodBon15} the authors consider a cartesian diagram
	\[
	\begin{tikzcd}
		\widehat{X} \ar[d, "p"'] \ar[r, "q"] & X_{+} \ar[d, "f_{+}"]\\
		X_{-} \ar[r, "f_{-}"] & Y
	\end{tikzcd}
	\]
	where $f_{-}$ is subject to various assumptions (among which there are $(f_{-})_{\ast} \ca{O}_{X_{-}} \simeq \ca{O}_Y$ and that the fibres of $f_{-}$ have dimension at most $1$), and $f_{+}$ is the flop of $f_{-}$.

	Under these assumption $X_{-}$ has a tilting bundle $M$ such that there exists a distinguished triangle, [\emph{ibidem}, Lemma 4.4, Proposition 4.12, Lemma 5.2],
	\begin{equation*}
		q^{\ast}q_{\ast}p^{\ast}M  \rightarrow p^{\ast} M \rightarrow \ca{H}^{0}(p^{\ast} P)[1].
	\end{equation*}
	Here $P$ is the projection of $M$ to the category $\ca{A}_{f_{-}} = \{ E \in \text{Coh}(X_{-})  : (f_{-})_{\ast} E = 0 \}$, and it is a projective generator for $\ca{A}_{f_{-}}$.
	Furthermore, Bodzenta and Bondal prove that
	\[
		\Hom^{\bullet}_{\dk}(\ca{H}^{0}(p^{\ast} P), \ca{H}^{0}(p^{\ast} P) ) \simeq \mathrm{RHom}_{\ca{A}_{f_{-}}}(P,P) =: A_{P},
	\]
	which is an algebra as $P$ is projective.
	Thus, by \autoref{lem:compact-generation-s} and \autoref{thm:4-sods-quotient} we get that $\ker \overline{p}_{\ast}$ is compactly generated by $\ca{H}^{0}(p^{\ast} P)[1]$ and $\derived(A_{P}) \simeq \ker \overline{p}_{\ast}$.
	By \autoref{thm:relation-to-bounded-case} and results in \cite{BodBon15}, restricting to $\derived^{\bounded}(\widehat{X}) /\mcr{K}^{\bounded}$ we get $\ker \overline{p}_{\ast} \cap \derived^{\bounded}(\widehat{X}) /\mcr{K}^{\bounded} \simeq \derived^{\bounded}(\ca{A}_{f_{-}})$.

	Under the futher assumption that $Y = \text{Spec} \, R$, $R$ a complete local $k$-algebra, Bodzenta and Bondal prove that $A_{P} \simeq A_{\mathrm{con}}$, the contraction algebra as defined in \cite{Donovan-Wemyss-Contraction-algebra}.
	Therefore, in this case we get $\ker \overline{p}_{\ast} \simeq \derived(A_{\mathrm{con}})$. 
\end{rmk}

\section{Examples}
\label{sect:examples}
We now work through some examples where we can describe the category $\ker \overline{p}_{\ast}$ explicitly.
From now on $k$ denotes an algebraically closed field of characteristic zero.

We will apply \autoref{thm:spherical-functor-from-sods} in the following setup: we consider a roof $X_{-} \xleftarrow{p} \widehat{X} \xrightarrow{q} X_{+}$ of separated, finite type schemes of finite Krull dimension together with finite type maps; furthermore, we assume $p_{\ast} \ca{O}_{\widehat{X}} \simeq \ca{O}_{X_{-}}$, $q_{\ast} \ca{O}_{\widehat{X}} \simeq \ca{O}_{X_{+}}$, and that the functors $p_{\ast} q^{\ast}$, $q_{\ast}p^{\ast}$ are equivalences between $\derived_{\qc}(X_{\pm})$ and $\derived_{\qc}(X_{\mp})$.
Then, diagram \eqref{eqn:std-diagram} is given by $\derived_{\qc}(X_{-}) \xleftarrow{p_{\ast}} \derived_{\qc}(\widehat{X}) \xrightarrow{q_{\ast}} \derived_{\qc}(X_{+})$.

\subsection{Standard flops (local model)}
\label{section:std-flops}
Let $X_{-} = X_{+} = \Tot( \ca{O}_{\mb{P}^{n}}(-1)^{\oplus n+1})$ and consider
\[
	R = H^0(X_{\pm}, \ca{O}_{X_{\pm}} ) = \bigoplus_{r \geq 0} H^0(\mb{P}^n, \text{Sym}^r \ca{O}_{\mb{P}^n}(1)^{\oplus n+1}).
\]
Then, if we set $Y = \text{Spec}(R)$ and $\widehat{X} = \Tot(\ca{O}_{\mb{P}^n \times \mb{P}^n}(-1,-1)) \simeq \Bl_{\mb{P}^n}(X_{\pm})$, we have a cartesian diagram
\begin{equation}
\label{eqn:cartesian-diagram}
	\begin{tikzcd}
		\widehat{X} \ar[d, "p"'] \ar[r, "q"] & X_{+} \ar[d, "f_{+}"]\\
		X_{-} \ar[r, "f_{-}"] & Y
	\end{tikzcd}
\end{equation}
where $f_{-}$ and $f_{+}$ are the affinization maps.
To ease the notation we will denote $\mb{P}^n = \mb{P}$.

Applying \cite[Lemma 2.12]{KL15} to the equivalences of \cite{Bondal-Orlov-SOD-alg-var} we get that $p_{\ast} q^{\ast}$ and $q_{\ast} p^{\ast}$ induce equivalences between $\derived_{\qc}(X_{\pm})$ and $\derived_{\qc}(X_{\mp})$.
Therefore, we can apply\footnote{Notice that $p$ and $q$ are also of finite Tor dimension because $X_{\pm}$ are smooth, and they are proper as they are blow-ups of finite type ideals.} \autoref{thm:spherical-functor-from-sods} and we get the conservative spherical functor $\overline{\Psi}_{+} : \ker \overline{p}_{\ast} \rightarrow \dqc$ whose inverse twist is the autoequivalence $q_{\ast}p^{\ast}p_{\ast}q^{\ast}$.

Our aim is now to give an explicit description of $\ker \overline{p}_{\ast}$.
Before going into the formal argument, let us make an educated guess.
In \cite{Addington-Donovan-Meachan} the authors prove that the object $\ca{O}_{\mb{P}} \in \derived_{\qc}(X_{+})$ is spherical \cite[Def. 1.1]{Seidel-Thomas01} and that we have an isomorphism of functors $q_{\ast} p^{\ast} p_{\ast} q^{\ast} \simeq T_{\ca{O}_{\mb{P}}(-1)}^{-1}  T_{\ca{O}_{\mb{P}}(-2)}^{-1}\dots  T_{\ca{O}_{\mb{P}}(-n)}^{-1}$.
Considering the gluing procedure defined in \cite{Barb-Spherical-twists}, we would expect $\ker \overline{p}_{\ast}$ to have a full, exceptional collection consisting of $n$ elements such that the image of the $i$-th exceptional object (counting right to left in the SOD) under $\overline{\Psi}_{+}$ is given by $\ca{O}_{\mb{P}}(-i)$.
We will indeed prove this guess true, see \autoref{thm:C=D(R)} and \autoref{cor:exceptional-coll-std-flops}.

\begin{dfn}
	Given a category $\mcr{A}$ and objects $a_1, \dots, a_k \in \mcr{A}$ we define the algebra $\Hom_{\mcr{A}}^{\rightarrow}(\left\{a_j\right\})$ as the sub-algebra of $\Hom_{\mcr{A}}(\oplus_{j=1}^k a_j, \oplus_{j=1}^k a_j)$ generated by the identities and the morphisms going from $a_i$ to $a_j$ for $i > j$.
\end{dfn}

\begin{rmk}
	If $\mcr{A}$ is a dg-category, then $\Hom_{\mcr{A}}^{\rightarrow}(\left\{a_j\right\})$ is a sub-dg-algebra and it enjoys the following property
	\begin{equation}
		\label{property-sub-dgas}
		\begin{array}{lcr}
			d(s) \in \Hom_{\mcr{A}}^{\rightarrow}(\left\{a_j\right\}) \iff s \in \Hom_{\mcr{A}}^{\rightarrow}(\left\{a_j\right\}) & & \forall s \in \Hom_{\mcr{A}}(\oplus_{j=1}^k a_j, \oplus_{j=1}^k a_j)
		\end{array}
	\end{equation}
	where $d(-)$ is the differential.
	As a consequence, we get an inclusion $H^{\bullet}(\Hom_{\mcr{A}}^{\rightarrow}(\left\{a_j\right\})) \hookrightarrow H^{\bullet}(\Hom_{\mcr{A}}(\oplus_{j=1}^k a_j, \oplus_{j=1}^k a_j))$.
\end{rmk}

Let us denote $\inj_{X_{+}}$ the dg-category of h-injective complexes on $X_{+}$, see \cite{Resol-unbounded-complexes-Spaltenstein}.
Then, we have

\begin{thm}
	\label{thm:C=D(R)}
	There exist h-injective resolutions $I_1, \dots, I_n$ of $\ca{O}_{\mb{P}}(-1), \dots, \ca{O}_{\mb{P}}(-n)$ such that we have an equivalence $\ker \overline{p}_{\ast} \simeq \derived (\Hom_{\inj_{X_{+}}}^{\rightarrow}(\left\{I_j\right\}))$ and under this equivalence the spherical functor $\overline{\Psi}_{+}$ is identified with $- \stackrel{L}{\otimes}_{\Hom_{\inj_{X_{+}}}^{\rightarrow}(\left\{I_j\right\})} I_n \oplus \dots \oplus I_{1}$.
\end{thm}

\begin{rmk}
	Notice that $\gen{\ca{O}_{\mb{P} \times \mb{P}}(0,-1) \oplus \dots \oplus \ca{O}_{\mb{P} \times \mb{P}}(0,-n)} \subset \derived_{\qc}(\widehat{X})$ is not the correct source category for the spherical functor realising the flop-flop autoequivalence except in the case $n = 1$.
	Indeed, if we take $n=2$, then $\Hom^{\bullet}_{\derived_{\qc}(\widehat{X})}(\ca{O}_{\mb{P}\times \mb{P}}(0,-2), \ca{O}_{\mb{P}\times \mb{P}}(0,-1)) \simeq k^{3}$ but what we should get considering \cite{Barb-Spherical-twists} is
	\[
		\Hom^{\bullet}_{\derived_{\qc}(X_{+})}(\ca{O}_{\mb{P}}(-2), \ca{O}_{\mb{P}}(-1)) \simeq k^{3} \oplus k^{3}[-1].
	\]
\end{rmk}

\begin{proof}
	\label{proof:thm-std-flops}
	By \autoref{lem:cpt-gen-std-flops} the object $Q(\oplus_{j=1}^n \ca{O}_{\mb{P} \times \mb{P}}(0,-j))$ is a compact generator for $\ker \overline{p}_{\ast}$.
	Then, by \autoref{thm:4-sods-quotient} we get that 
	\begin{equation}
		\label{eqn:proof-std-flops-algebra-iR-gen}
		\End^{\bullet}_{\derived_{\qc}(\widehat{X})}(i_{\mcr{S}_{+}}^R (\oplus_{j=1}^n \ca{O}_{\mb{P} \times \mb{P}}(0,-j) )) \simeq \End^{\bullet}_{\dk}(Q(\oplus_{j=1}^n \ca{O}_{\mb{P} \times \mb{P}}(0,-j))).
	\end{equation}

 	Fix compact generators $E \in \derived_{\qc}(X_+)$ and $F \in \derived_{\qc}(\widehat{X})$, choose h-injective resolutions $I_E$ and $I_F$, and denote $R_E$ and $R_F$ their endomorphism dg-algebras.
	By \cite[Proposition 1.17]{Lunts-Orlov-Uniqueness-enhancement} we get equivalences $\Xi_1 \colon \derived(R_E) \xrightarrow{\simeq} \derived_{\qc}(X_{+})$, $\Xi_2 \colon \derived(R_F) \xrightarrow{\simeq} \derived_{\qc}(\widehat{X})$.

	Then, set $\ca{P}(R_E)$ and $\ca{P}(R_F)$ to be the dg-categories of h-projective dg-modules over $R_E$ and $R_F$, respectively.
	As it is explained in \cite[Example 4.2]{Anno-Log-17}, the functor $q_{\ast} \colon \derived_{\qc}(\widehat{X}) \rightarrow \derived_{\qc}(X_{+})$ admits a lift to a quasi-functor between $\ca{P}(R_F)$ and $\ca{P}(R_E)$.
	What this means is that if we fix a cofibrant replacement $\ca{A}$ for $\ca{P}(R_F)$ in the model category structure of $\mathrm{dgcat}_{k}$ defined in \cite{Tabuada-DGCAT}, then we can construct a diagram of dg-categories
	\begin{equation}
		\label{eqn:roof-standard-flops}
		\ca{P}(R_F) \overset{\alpha}{\longleftarrow} \ca{A} \overset{\beta}{\longrightarrow} \ca{P}(R_E)
	\end{equation}
	such that $\alpha$ is a quasi-equivalence, \emph{i.e.}, it induces quasi-isomorphisms between the complexes of morphisms and $H^{0}(\alpha)$ is essentially surjective, and
	\begin{equation}
		\label{isom-dg-lift}
		\Xi_1  \circ H^0(\beta) \circ H^0(\alpha)^{-1} \circ \Xi_2^{-1} \simeq q_{\ast}.
	\end{equation}

	Take $A_i \in \ca{A}$ an object such that $\Xi_2 ( H^0(\alpha)(A_i) ) \simeq i_{\mcr{S}_{+}}^R(\ca{O}_{\mb{P} \times \mb{P}}(0,-i))$ and set $A = \bigoplus_{i=1}^n A_i$.
	Then, we get a diagram of dg-algebras and morphism of dg-algebras
	\begin{equation}
		\label{eqn:diagram-dga-proof-standard-flops}
		\Hom_{\ca{P}(R_F)}(\alpha(A), \alpha(A)) \overset{\alpha}{\longleftarrow} \Hom_{\ca{A}}(A,A) \overset{\beta}{\longrightarrow} \Hom_{\ca{P}(R_E)}(\beta(A), \beta(A)),
	\end{equation}
	where $\alpha$ is a quasi-isomorphism.
	Notice that by the choice of the $A_i$ and \eqref{eqn:proof-std-flops-algebra-iR-gen} we have
	\begin{equation}
		\label{eqn:algebras-a}
		H^{\bullet}(\Hom_{\ca{P}(R_F)}(\alpha(A),\alpha(A))) \simeq H^{\bullet}(\Hom_{\ca{A}}(A,A)) \simeq \End^{\bullet}_{\dk}(\oplus_{i=1}^{n} Q(\ca{O}_{\mb{P} \times \mb{P}}(0,-i))).
	\end{equation}
	Thus, by \eqref{property-sub-dgas} and \autoref{lem:end-algebra-std-flops} we get that the inclusions
	\[
		\begin{array}{lcr}
			\Hom_{\ca{P}(R_F)}^{\rightarrow}(\{ \alpha(A_i)\}) \hookrightarrow \Hom_{\ca{P}(R_F)}(\alpha(A), \alpha(A)) & \mathrm{and} & \Hom_{\ca{A}}^{\rightarrow}(\{A_i\}) \hookrightarrow \Hom_{\ca{A}}(A,A)
		\end{array} 
	\]
	are quasi-isomorphisms.
	Moreover, diagram \eqref{eqn:diagram-dga-proof-standard-flops} gives rise to a diagram
	\[
		\Hom_{\ca{A}}^{\rightarrow}(\{ \alpha(A_i)\}) \overset{\alpha}{\longleftarrow} \Hom_{\ca{A}}^{\rightarrow}(\{A_i\}) \overset{\tilde{\beta}}{\longrightarrow} \Hom_{\ca{P}(R_E)}^{\rightarrow}(\{\beta(A_i)\}) \overset{\iota}{\hookrightarrow} \Hom_{\ca{P}(R_E)}(\beta(A), \beta(A))
	\]
	where $\alpha$ is a quasi-isomorphism, $\beta = \iota  \tilde{\beta}$, and $\tilde{\beta}$ is a quasi-isomorphism by \eqref{eqn:algebras-a} and \autoref{lem:end-algebra-std-flops} because by \eqref{isom-dg-lift} we have $\Xi_1 (H^0(\beta)(A_i)) \simeq \ca{O}_{\mb{P}}(-i)$.

	To conclude, recall that, as it is explained in \cite[Remark 1.1]{Lunts-Orlov-Uniqueness-enhancement}, the equivalence $\Xi_2$ comes from a quasi-functor $\xi \colon \inj_{X_{+}} \rightarrow \ca{S}\ca{F}(R_E)$ where $\ca{S}\ca{F}(R_E)$ is the dg-category of \emph{semifree} dg-modules.
	Moreover, recall that the inclusion $\ca{S} \ca{F}(R_E) \hookrightarrow \ca{P}(R_E)$ is a quasi-equivalence, see \cite[ C.8]{Drinfeld-DG-quotients}.
	Fixing a cofibrant replacement of $\ca{P}(R_E)$ and proceeding as we did above, we can realise the quasi-functor $\xi$ as a roof of dg-functors similar to \eqref{eqn:roof-standard-flops}, the only difference being that now both arrows in the diagram are quasi-equivalences.
	Hence, we see that there exist h-injective resolutions $I_j$'s of the $\ca{O}_{\mb{P}}(-j)$'s such that we have a quasi-isomorphism $\Hom_{\ca{P}(R_E)}^{\rightarrow}(\{\beta(A_i)\}) \underset{\mathrm{qis}}{\simeq} \Hom_{\inj_{X_{+}}}^{\rightarrow}(\{ I_j\})$.

	Putting together all the quasi-isomorphisms we constructed, we get an equivalence $\ker \overline{p}_{\ast} \simeq \derived (\Hom_{\inj_{X_{+}}}^{\rightarrow}(\{ I_j\}))$.
	Hence, to conclude we only have to prove the description of the spherical functor.
	The statement follows because the dg-module $\Hom_{\inj_{X_{+}}}^{\rightarrow}(\{ I_j\})$ is mapped to $I_n \oplus \dots \oplus I_1$.
\end{proof}

\begin{cor}
	\label{cor:exceptional-coll-std-flops}
	The category $\ker \overline{p}_{\ast}$ has a full, exceptional collection of length $n$, \emph{i.e.},
	\[
		\ker \overline{p}_{\ast} = \gen{\derived(k), \dots, \derived(k)}.
	\]
	Furthermore, the right gluing bimodule between the $j$-th and the $i$-th component (counting right to left) is $\mathrm{RHom}_{X_{+}}(\ca{O}_{\mb{P}}(-j), \ca{O}_{\mb{P}}(-i))$.
\end{cor}

\begin{proof}
	Both statements follow from \cite[Prop. 4.6]{KL15} considering \cite[Remark 2.6.5]{Barb-Spherical-twists} because the dg-algebra in \autoref{thm:C=D(R)} is upper triangular.
\end{proof}

\subsection{Standard flops in families}
\label{section:std-flops-families}
We now want to generalise the previous example.
Until now, we have put ourselves in a local situation.
Namely, we had the total space of a vector bundle and $\mb{P} = \mb{P}^n$ was the zero section.
We have two ways to generalise this setup.
First, we can take $X_{-}$ to be any variety.
Then, $\mb{P}$ would be embedded as a closed subscheme such that $N_{\mb{P}/X_{-}} = \ca{O}_{\mb{P}}(-1)^{n+1}$.
Second, we could consider
\begin{equation}
\label{eqn:standard-flop-families}
	\begin{tikzcd}
		\mb{P}V_{-} \ar[d, "\omega_{-}"] \ar[r, hook, "j_{-}"] & X_{-}\\
		Z,
	\end{tikzcd}
\end{equation}
where $Z$ is a smooth complex variety, $V_{-}$ is a vector bundle or rank $n+1$ over $Z$, and $j_{-}$ is a closed embedding with normal bundle $N_{\mb{P}V_{-}/X_{-}} = \ca{O}_{\mb{P}V_{-}}(-1) \otimes \omega_{-}^{\ast}V_{+}$ for some vector bundle $V_{+}$ of rank $n+1$ over $Z$.\footnote{The first generalisation is the special case $Z = \mathrm{pt}$. Notice that when $Z$ is general we have a family version of the case $Z = \mathrm{pt}$: we are simultaneously flopping a family of projective spaces over the base $Z$.}
Then, one can flop $X_{-}$ along $j_{-}(\mb{P}V)$, see \cite{Huy06}, obtaining the following
\[
	\begin{tikzcd}
		{} & {} & \ar[lldd, "p"'] \widehat{X} \ar[rrdd, "q"]& {} & {}\\
		{} & {} & \ar[dl, "\eta_{-}"] E = \mb{P}V_{-} \times_Z \mb{P}V_{+} \ar[dr, "\eta_{+}"'] \ar[u, "i", hook] & {} & {}\\
		X_{-} & \ar[l, "j_{-}",hook'] \mb{P}V_{-} \ar[rd, "\omega_{-}"] & {} & {} \ar[dl, "\omega_{+}"'] \mb{P}V_{+} \ar[r, "j_{+}"', hook] & X_{+}\\
		{} & {} & Z. & {} & {}
	\end{tikzcd}
\]
The hypotheses of \autoref{thm:spherical-functor-from-sods} are satisfied\footnote{Notice that we have more: as in the local case, $p$ and $q$ are of finite tor dimension as $X_{\pm}$ are smooth, and they are proper as they are blow-ups in finite type ideals.} because we can apply \cite[Lemma 2.12]{KL15} and deduce the equivalence $\derived_{\qc}(X_{-}) \simeq \derived_{\qc}(X_{+})$ (via the flop functors) from the equivalences between the bounded derived categories of coherent sheaves.

Consider the fully faithful functors
\[
	\begin{array}{c}
		I_{a,b}(-) := i_{\ast} \left( \ca{O}_{E}(a,b) \otimes \eta_{-}^{\ast} \omega_{-}^{\ast}(-) \right) \colon \derived_{\qc}(Z) \rightarrow \derived_{\qc}(\widehat{X})\\
		\Phi_{k}(-) := i_{\ast} \left( \ca{O}_E(kE) \otimes \eta_{-}^{\ast}(-) \right) \colon \derived_{\qc}(\mb{P}V_{-}) \rightarrow \derived_{\qc}(\widehat{X})
	\end{array}.
\]
Notice that by Grothendieck--Verdier duality $I_{a,b}$ has a left adjoint given by
\[
	I_{a,b}^L(-) = (\omega_{-})_{\ast} (\eta_{-})_{\ast} \left( \omega_{E/Z}[\dim E - \dim Z] \otimes \left( \ca{O}_{E}(-a,-b) \otimes i^{\ast}(-) \right) \right)
\]
where $\omega_{E/Z}$ is computed relative to the map $\eta_{-}  \omega_{-} = \eta_{+}  \omega_{+}$.

In \cite{Addington-Donovan-Meachan}, the authors prove that we have an isomorphism of functors $q_{\ast}p^{\ast}p_{\ast}q^{\ast} \simeq T^{-1}_{J_{-1}}  \dots  T^{-1}_{J_{-n}}$, where $J_{-i} : \derived_{\qc}(Z) \rightarrow \derived_{\qc}(X_{+})$ is the spherical functor $J_{-i}(F) = (j_{+})_{\ast}( \ca{O}_{\mb{P}V_{+}}(-i) \otimes \omega^{\ast}_{+}F)$.
Similarly to the local case, we will prove that $\ker \overline{p}_{\ast}$ is the gluing of $n$ copies of $\derived_{\qc}(Z)$ along the functors $J_{-i}$.

\begin{rmk}
	In \cite{Addington-Donovan-Meachan}, at the end of  2, the authors point out that the flop-flop functor should fit into the framework of \cite[Theorem 3.11]{Halpern-Shipman16}.
	Our four periodic SOD implements this framework.
\end{rmk}

Recall the definition of the right mutation along $\im(I_{a,b})$ as an endofunctor of $\derived_{\qc}(\widehat{X})$: $\mb{R}_{a,b} = \cone\left( \id \rightarrow I_{a,b} I^L_{a,b} \right)[-1]$.

The blow-up formula together with \cite[Lemma 3.20]{Halpern-GIT-15} tell us that we have a SOD
\[
	\derived_{\qc}(\widehat{X}) = \langle \im(\Phi_{n}), \dots, \im(\Phi_{2}), \im(\Phi_{1}), p^{\ast} \derived_{\qc}(X_{-}) \rangle.
\]
Moreover, noticing that $\ca{O}_{E}(kE) \simeq \ca{O}_{\mb{P}V_{-} \times_{Z} \mb{P}V_{+}}(-k,-k)$, we can use Orlov's semiorthogonal decomposition of projective bundles to get that
\[
	\im(\Phi_{k}) = \langle \im(I_{-k+a, -k}), \dots, \im(I_{-k+n+a, -k})\rangle \quad a \in \mathbb{Z}.
\]
Putting these two things together, we get
\begin{equation}
	\label{SOD-roof-non-mutated}
		\derived_{\qc}(\widehat{X}) =
		\langle
		\im(I_{-n,-n}), \dots, \im(I_{0,-n}),
		\dots,
		\im(I_{-n,-1}), \dots, \im(I_{0,-1}),
		p^{\ast}\derived_{\qc}(X_{-})
		\rangle
\end{equation}
Set $\mb{R}_{0} = \id$ and for $k = -1, \dots, -n+1$ define the functors
\[
		\mb{R}_{k} := \mb{R}_{-1,-1}  \mb{R}_{-2,-1}  \dots  \mb{R}_{-n,-1}  \dots  \mb{R}_{-1,k}  \mb{R}_{-2,k}  \dots  \mb{R}_{-n,k}.
\]
Then, we have the SOD (read each block top to bottom, left to right)
\begin{equation}
	\label{SOD-roof-families}
	\derived_{\qc}(\widehat{X}) = \langle 
		\underbrace{\begin{array}{l}
		\im(I_{-n,-n}), \dots, \im(I_{-1,-n}),\\
		\im(I_{-n,-n+1}), \dots \im(I_{-1,-n+1}),\\
		\dots\\
		\im(I_{-n,-1}), \dots, \im(I_{-1,-1}),
		\end{array}}_{\mcr{A}}
		\underbrace{\begin{array}{l}
		\mb{R}_{-n+1}\im(I_{0,-n}),\\
		\mb{R}_{-n+2} \im(I_{0,-n+1}),\\
		\dots,\\
		\mb{R}_{-1} \im(I_{0,-2}),\\
		\im(I_{0,-1}),
		\end{array}}_{\mcr{B}}
		p^{\ast}\derived_{\qc}(X_{-})
	\rangle.
\end{equation}

\begin{thm}
	\label{thm:std-flop-family-version}
	We have a SOD induced by the quotient functor $Q : \derived_{\qc}(\widehat{X}) \rightarrow \dk$
	\[
		\begin{tikzcd}
			\langle \mb{R}_{-n+1}\im(I_{0,-n}),
			\mb{R}_{-n+2} \im(I_{0,-n+1}),
			\dots,
			\mb{R}_{-1} \im(I_{0,-2}),
			\im(I_{0,-1}) \rangle \ar[r, "Q", "\simeq"'] & \ker \overline{p}_{\ast}
		\end{tikzcd}
	\]
\end{thm}

\begin{proof}
	As by \autoref{lem:K-in-SOD-appendix} $\mcr{A} = \mcr{K}$, this follows from \autoref{thm:4-sods-quotient} and SOD \eqref{SOD-roof-families}.
\end{proof}

\begin{ex}
	When $Z = \mathrm{pt}$ and $X_{\pm} = \Tot(\ca{O}_{\mb{P}}(-1)^{\oplus n+1})$ the SOD of \autoref{thm:std-flop-family-version} recovers the SOD of \autoref{section:std-flops}.
	Take $n=2$ for simplicity; then
	\[
		\renewcommand{\arraystretch}{1.5}
		\begin{array}{c}
			\mcr{K} = \mcr{A} = \gen{\ca{O}_{\mb{P} \times \mb{P}}(-2,-2), \ca{O}_{\mb{P} \times \mb{P}}(-1,-2), \ca{O}_{\mb{P} \times \mb{P}}(-2,-1), \ca{O}_{\mb{P} \times \mb{P}}(-1,-1)},\\
			\mcr{B} = \gen{\mb{R}_{-1}\ca{O}_{\mb{P} \times \mb{P}}(0,-2), \ca{O}_{\mb{P} \times \mb{P}}(0,-1)},
		\end{array}
	\]
	and a simple computation shows 
	\[
		\mb{R}_{-1}\ca{O}_{\mb{P} \times \mb{P}}(0,-2) \simeq
		\{
			\ca{O}_{\mb{P} \times \mb{P}}(0,-1)^{\oplus 3} \rightarrow \ca{O}_{\widehat{X}}(0,-2) \otimes \ca{O}_{\widehat{X}} \left/ \ca{I}_{\mb{P} \times \mb{P}}^{\, 2} \right.
		\}.
	\]
	Thus, we recover the resolution we find in \autoref{lem:resolution}.
\end{ex}

As we want to obtain the description of $\ker \overline{p}_{\ast}$ as a gluing of copies of $\derived_{\qc}(Z)$, we have to compute natural transformations between the various functors $\mb{R}_{-j+1}  I_{0,-j}$.

In approaching such a question we have to be careful because technicalities about dg-enhancements and dg-lifts of functors play a fundamental role.
For this reason, as all the functors we are interested in are Fourier--Mukai functors, instead of speaking of natural transformation between the functors, we speak about morphisms between the kernels.

Let us denote the Fourier--Mukai kernel for $I_{a,b}$ as $\ca{I}_{a,b}$, the one for $\mb{R}_{-j}$ as $\ca{R}_{-j}$, and the one for $J_{-i}$ as $\ca{J}_{-i}$.
For simplicity, we will denote $\ca{R}_{-j+1}\ca{I}_{0,-j}$ the kernel for $\mb{R}_{-j+1}  I_{0,-j}$.

The desired description of $\ker \overline{p}_{\ast}$ is a consequence of the following

\begin{prop}
	The functors $Q  \mb{R}_{-i+1}  I_{0,-i} \colon \derived_{\qc}(Z) \rightarrow \dk$ are fully faithful for $1 \leq i \leq n$.
	Moreover, for $1 \leq j < i \leq n$ we have an isomorphism 
	\[
		\mathrm{RHom}_{Z \times \widehat{X}}(\ca{R}_{-i+1} \ca{I}_{0,-i}, \ca{R}_{-j+1} \ca{I}_{0,-j}) \xrightarrow{(\id \times q)_{\ast}} \mathrm{RHom}_{Z \times X_{+}}(\ca{J}_{-i}, \ca{J}_{-j}).
	\]
\end{prop}

\begin{proof}
	The statement about fully faithfulness is equivalent to the fully faithfulness of $\mb{R}_{-i+1}  I_{0,-i}$ by \autoref{thm:std-flop-family-version}.
	However, as the mutations $\mb{R}_{a,b}$ are equivalences from $\im(I_{a,b})^{\perp}$ to ${}^{\perp}\im(I_{a,b})$, the statement reduces to the fully faithfulness of $I_{0,-i}$, which is true.
	For the isomorphism, see \autoref{lem:natural-transf-family-case-appendix}.
\end{proof}

\subsection{Mukai flops}
\label{section:Mukai-flops}
We keep using the notation introduced in \autoref{section:std-flops}.
We consider the function $f_{-}( {\bf x}, {\bf t} ) = \sum_{i=1}^{n+1} x_i t_i$, $f_{-} \in \ca{O}_{X_{-}}$, where the $x$'s are the base coordinates, the $t$'s are fibre coordinates, and we use the shorthand notation ${\bf x} = (x_1, \dots, x_{n+1})$, ${\bf t} = (t_1, \dots, t_{n+1})$.
Then, we define $Z_{-} = \{ f_{-} = 0 \}$.
Notice that we have $Z_{-} \simeq \Tot( \Omega^1_{\mb{P}^n})$.
The procedure of blowing-up $\mb{P}^n \subset X_{-}$ and then contracting the other $\mb{P}^n$ can be carried out on $Z_{-}$ as well, \cite{Kawamata-K-D-equivalence}, \cite{Namikawa-mukai-flops}.
We get the following diagram whose top row takes the name of \emph{Mukai flop}
\begin{equation}
	\label{eqn:diagram-mukai}
	\begin{tikzcd}[row sep = 3 em, column sep = 4 em]
		Z_{-} \ar[d, hook,  "f_{-} = 0"'] & \ar[l, "p"']  \widehat{Z} \ar[d, hook, "\hat{f} = 0"] \ar[r, "q"] & Z_{+} \ar[d, hook,  "f_{+} = 0"]\\
		X_{-}  & \ar[l, "p"] \widehat{X} \ar[r, "q"'] & X_{+}
	\end{tikzcd}.
\end{equation}
Here\footnote{Here $s$'s are fibre coordinates for $X_{+}$ and $u$ is the fibre coordinate of $\widehat{X}$.} $f_{+}( {\bf y}, {\bf s} ) = \sum_{i=1}^{n+1} y_i s_i$, $f_{+} \in \ca{O}_{X_{+}}$, and $\hat{f}({\bf x}, {\bf y }, u) = \sum_{i = 1}^{n+1} x_i y_i u$, $\hat{f} \in \ca{O}_{\widehat{X}}$.

Notice that the equation $\hat{f} = 0$ describes $\widehat{Z}$ as a normal crossing divisor in $\widehat{X}$.
Moreover, $\widehat{Z}$ has two irreducible components: the blow-up of $Z_{-}$ (or $Z_{+}$) along $\mb{P}^n$, that we denote $\widetilde{Z}$, and $\mb{P}^n \times \mb{P}^n$.
These two irreducible components are glued along $\mb{P} (\Omega^1_{\mb{P}^n})$ that is the exceptional locus of the blow-up and that sits inside $\mb{P}^n \times \mb{P}^n$ via the inclusion $\Omega^1_{\mb{P}^n} \hookrightarrow \ca{O}_{\mb{P}^n}^{\oplus n+1}(-1)$.
As we know that $\mb{P}^n \times \mb{P}^n \subset \widehat{X}$ is cut out by $\{u = 0\}$ in $\widehat{X}$, we get that $\widetilde{Z}$ is cut out by $\left\{ w := \sum_{i=1}^{n+1} x_i y_i = 0 \right\}$, $w \in \ca{O}_{\widehat{X}}(1,1)$.
Therefore, we have exact sequences
\begin{equation}
	\label{resol-P}
	\ca{O}_{\widetilde{Z}}(1,1) \xrightarrow{u} \ca{O}_{\widehat{Z}} \rightarrow \ca{O}_{\mb{P}^n \times \mb{P}^n}
\end{equation}
\begin{equation}
	\label{resol-Z}
		\ca{O}_{\mb{P}^n \times \mb{P}^n}(-1,-1) \xrightarrow{w} \ca{O}_{\widehat{Z}} \rightarrow  \ca{O}_{\widetilde{Z}}.
\end{equation}
The smooth varieties $Z_{-}$ and $Z_{+}$ are birational and derived equivalent via the functors $p_{\ast}q^{\ast}$ and $q_{\ast}p^{\ast}$ \cite{Kawamata-K-D-equivalence}.
Notice that the equivalence $\derived^{\bounded}(Z_{+}) \simeq \derived^{\bounded}(Z_{-})$ induces an equivalence $\derived_{\qc}(Z_{+}) \simeq \derived_{\qc}(Z_{-})$ by \cite[Lemma 2.12]{KL15}.
Therefore, \autoref{thm:spherical-functor-from-sods} applies and the flop-flop functor $q_{\ast}p^{\ast}p_{\ast}q^{\ast}$ is the inverse of the spherical twist around $\overline{\Psi}_{+} \colon \ker \overline{p}_{\ast} \rightarrow \derived_{\qc}(Z_+)$.

\begin{ex}[Where boundedness fails]
	\label{ex:where-bound-fails}
	We wish to show $\mcr{C} := \tgen{\im(Q  \pi_{-}  q^{\ast} \vert_{\derived^{\bounded}(Z_{+})})} \subsetneq \ker \overline{p}_{\ast} \subset \derived^{\bounded}(\widehat{Z}) / \mcr{K}^{\bounded}$ thus proving that passing to $\derived_{\qc}(-)$ was indeed a necessary step.

	Let us consider the Mukai flop for $n=1$, we will show $Q(\ca{O}_{\mb{P} \times \mb{P}}(0,-1)) \notin \mcr{C}$.
	Notice that as the relative dimension of the fibres is at most one, by \autoref{subsect:bounded-derived-categories} we have a fully faithful embedding $\derived^{\bounded}(\widehat{Z}) / \mcr{K}^{\bounded} \hookrightarrow \derived_{\qc}(\widehat{Z})/\mcr{K}$ and therefore $\im(Q  \pi  q^{\ast} \vert_{\derived^{\bounded}(Z_{+})})$ generates the same subcategory when considered inside either categories.
	$\derived^{\bounded}(Z_{+})$ is generated by $\ca{O}_{Z_{+}} \oplus \ca{O}_{Z_{+}}(-1)$, thus by \autoref{thm:4-sods-quotient} $\ker \overline{p}_{\ast} $ is generated by $\overline{F} = Q \left( F \right)$, $F = \text{cone} \left(	\ca{O}_{\widehat{Z}} (1,0) \overset{u}{\rightarrow} \ca{O}_{\widehat{Z}} (0,-1) \right)$.
	Using \eqref{resol-P} and \eqref{resol-Z} we see that $\ca{H}^i (F) \simeq \ca{O}_{\mb{P} \times \mb{P}}(0,-1)$ for $i= -1, -2$ and zero otherwise.
	Hence, we get
	\[
		\text{End}^{\bullet}_{\derived_{\qc}(\widehat{Z})/\mcr{K}}(\overline{F}) \simeq \text{End}^{\bullet}_{\derived_{\qc}(\widehat{Z})}(F) \simeq k[\varepsilon]/\varepsilon^2,
	\]
	where $\varepsilon$ has degree $-1$ and the first isomorphism follows from \cite[Lemma 5.7]{BodBon15}.
	As $k[\varepsilon]/\varepsilon^2$ is \emph{intrinsically formal}, \emph{i.e.}, every dga with the same cohomology is quasi-isomorphic to it, by \cite[Prop. 1.17]{Lunts-Orlov-Uniqueness-enhancement} we get $\ker \overline{p}_{\ast}  \simeq \derived(k[\varepsilon]\left/ (\varepsilon^2)\right.)$.
	If $Q(\ca{O}_{\mb{P} \times \mb{P}}(0,-1)) \in \mcr{C} = \tgen{\overline{F}}$, then via this equivalence it should correspond to a compact object.
	However, $Q(\ca{O}_{\mb{P} \times \mb{P}}(0,-1))$ maps to the module $k$ concentrated in degree $0$, and this module is not compact.
	Hence, $Q(\ca{O}_{\mb{P} \times \mb{P}}(0,-1)) \notin \mcr{C}$.
	Notice however that
	\begin{equation}
		\label{eqn:hocolim-description}
		\mathrm{hocolim} \left\{ \right. \dots \rightarrow \overline{F}[2] \rightarrow \overline{F}[1] \rightarrow \overline{F} \left. \right\} \simeq Q(\ca{O}_{\mb{P} \times \mb{P}}(0,-1))
	\end{equation}
	and therefore $Q(\ca{O}_{\mb{P} \times \mb{P}}(0,-1)) \in \cgen{\overline{F}}  = \ker \overline{p}_{\ast}$, as it ought to be.
\end{ex}

Given a $\mb{P}^n$-object $E$ \cite[Def. 1.1]{Huyb-Thomas06} (or, more generally, a split $\mb{P}$-functor \cite{Add-symmetries.hyperkahler}, \cite{Cautis-Flops-and-about}, or a $\mb{P}$-functor \cite{AL-P-n-func}), one can construct an autoequivalence $P_E$ of $\derived^{\bounded}(X)$ called the $\mb{P}$-twist around the $\mb{P}^n$-object $E$ \cite{Huyb-Thomas06}.
By \cite[Lemma 2.12]{KL15} we have $P_E \in \Aut(\derived_{\qc}(X))$.

In \cite{Addington-Donovan-Meachan} the authors prove that $\ca{O}_{\mb{P}} \in \derived_{\qc}(Z_{+})$, $\mb{P} = \mb{P}^n \subset Z_{+}$, is a $\mb{P}^{n}$-object, and that $q_{\ast}p^{\ast}p_{\ast}q^{\ast} \simeq P_{\ca{O}_{\mb{P}}(-1)}^{-1}  \dots  P_{\ca{O}_{\mb{P}}(-n)}^{-1}$.

By \cite[Proposition 4.2]{Seg-autoeq-spherical-twist} $P_E$ is the spherical twist around the spherical functor
\[
	\derived(k[q]) \xrightarrow{- \stackrel{L}{\otimes}_{k[q]} E} \derived_{\qc}(X)
\]
where we assumed we replaced $E$ with an h-injective resolution and we make $q$ act via a representative $\widetilde{q}$ of the cohomology class of the extension $E \rightarrow E[2]$.
Thus, in analogy with \autoref{section:std-flops}, we would expect $\ker \overline{p}_{\ast} $ to have an SOD in terms of $n$ copies of $\derived(k[q])$.
However, this guess is wrong because $\widehat{Z}$ is singular.
More precisely, $\oplus_{j=1}^n Q(\ca{O}_{\mb{P} \times \mb{P}}(0,-j))$ generates $\ker \overline{p}_{\ast} $ but it is not compact.
Indeed, using \eqref{eqn:hocolim-description} and \cite[Lemma 5.7]{BodBon15} we get a quasi-isomorphism of complexes of vector spaces
\[
	\End^{\bullet}_{\dk}(Q(\ca{O}_{\mb{P} \times \mb{P}}(0,-1))) \simeq \bigoplus_{n \geq 0} k[-2n],
\]
and thus $Q(\ca{O}_{\mb{P} \times \mb{P}}(0,-1))$ cannot be compact.

To describe $\ker \overline{p}_{\ast}$ we use \cite[Remark 4.4]{Seg-autoeq-spherical-twist}: the functor
\[
	\derived(k[\varepsilon]/\varepsilon^2) \xrightarrow{-\stackrel{L}{\otimes}_{k[\varepsilon]/\varepsilon^2} I_E} \derived_{\qc}(X),
\]
where $I_E := \{ E[-1] \rightarrow E[1]\}$, is spherical and its spherical twist is given by $P_E$.

\begin{thm}
	\label{thm:C=D(R)-Mukai-complete}
	There exist a choice of h-injective resolutions $I_j$ of $\ca{O}_{\mb{P}}(-j)$ and of representatives $\tilde{q}_j$ of the degree two morphisms $q_j \in \End^{\bullet}_{\derived_{\qc}(Z_{+})}(\ca{O}_{\mb{P}}(-j))$ such that we have an SOD
	\[
		\ker \overline{p}_{\ast} = \gen{\derived(k[\varepsilon_n]/\varepsilon_n^2), \dots, \derived(k[\varepsilon_1]/\varepsilon_1^2)}
	\]
	with $\deg(\varepsilon_j) = -1$ and the right gluing bimodules for this SOD are given by $\Hom_{\inj_{Z_+}}(I_j', I_i')$ for $j \geq i$ where $I_j' = \left\{ I_j[-1] \rightarrow I_j[1] \right\}$ is the h-injective complex given by the extension $\tilde{q}_j$.
	Furthermore, we have $\overline{\Psi}_{+}(k[\varepsilon_j]/\varepsilon_j^2) \simeq I_j'$.
\end{thm}

\begin{proof}
	We prove the statement for $n=2$, the general case being analogous.
	Consider the generators $E_0 = \ca{O}_{Z_{+}}$, $E_1 = \ca{O}_{Z_{+}}(-1)$, $E_2 =\{ \ca{O}_{Z_{+}}(-1)^{\oplus 3} \xrightarrow{\mathbf{s}} \ca{O}_{Z_{+}}(-2)\}$ of $\derived_{\qc}(Z_{+})$ and take $\overline{E} = \oplus_{i=0}^2 \overline{\Psi}^L_{+}(E_i)$ the compact generator of $\ker \overline{p}_{\ast}$ (notice that $\overline{\Psi}^L_{+}(E_0) = 0$).

	Notice that, as $Z_{+} \subset X_{+}$ is an hypersurface cut out by a section $f_{+} \in \ca{O}_{X_{+}}$, we get the distinguished triangle
	\[
		\ca{O}_{\mb{P}}[1] \rightarrow \{ \ca{O}_{Z_{+}}(3) \rightarrow \ca{O}_{Z_{+}}(2)^{\oplus 3} \rightarrow \ca{O}_{Z_{+}}(1)^{\oplus 3} \rightarrow \ca{O}_{Z_{+}} \} \rightarrow \ca{O}_{\mb{P}}
	\]
	where the complex in the middle is the Koszul complex associated to $\mathbf{s}$, see e.g. \cite[Corollary 11.4]{Huy06}.
	Moreover, notice that we have long exact sequences
	\begin{equation}
		\label{eqn:psi-1}
		\ca{O}_{\mb{P}\times\mb{P}}(0,-1) \rightarrow \ca{O}_{\widehat{Z}}(1,0) \rightarrow \ca{O}_{\widehat{Z}}(0,-1) \rightarrow \ca{O}_{\mb{P}\times\mb{P}}(0,-1)
	\end{equation}
	\begin{equation}
		\label{eqn:psi-2}
		\ca{O}_{\widehat{Z}}(-2,-1) \rightarrow \ca{O}_{\widehat{Z}}(-1,-1)^{\oplus 3} \rightarrow \ca{O}_{\widehat{Z}}(0,-1)^{\oplus 3} \rightarrow \ca{O}_{\widehat{Z}}(1,-1)
	\end{equation}
	and the distinguished triangle
	\begin{equation}
		\label{eqn:dt-psi-2-first}
		\begin{aligned}
			\ca{O}_{\mb{P}\times\mb{P}}(0,-2)[1] & \, \rightarrow \{\ca{O}_{\widehat{Z}}(-2,-1) \rightarrow \ca{O}_{\widehat{Z}}(-1,-1)^{\oplus 3} \rightarrow \ca{O}_{\widehat{Z}}(0,-1)^{\oplus 3} \rightarrow \ca{O}_{\widehat{Z}}(0,-2) \} \\
			& \, \rightarrow \ca{O}_{\mb{P}\times\mb{P}}(0,-2)
		\end{aligned}
	\end{equation}
	The sequence \eqref{eqn:psi-1} is obtained joining \eqref{resol-P} and \eqref{resol-Z}, \eqref{eqn:psi-2} is the Koszul complex associated to $\mathbf{x}$, and \eqref{eqn:dt-psi-2-first} is obtained by tensoring \eqref{eqn:psi-1} with $\ca{O}_{\widehat{Z}}(0,-1)$ and replacing $\ca{O}_{\widehat{Z}}(1,-1)$ with the complex given by \eqref{eqn:psi-2}.

	Sequence \eqref{eqn:psi-1} can be rewritten as the distinguished triangle
	\begin{equation}
		\label{eqn:dt-psi-1}
		Q(\ca{O}_{\mb{P}\times\mb{P}}(0,-1))[1] \rightarrow \overline{\Psi}^{L}_{+}(\ca{O}_{Z_{+}}(-1)) \rightarrow Q(\ca{O}_{\mb{P}\times\mb{P}}(0,-1)),
	\end{equation}
	and thus we see $\End^{\bullet}_{\derived_{\qc}(Z_{+}) \left/ \mcr{K} \right.}(\overline{\Psi}^{L}_{+}(\ca{O}_{Z_{+}}(-1))) \simeq k[\varepsilon_1] / \varepsilon_1^2$, $\deg(\varepsilon_1) = -1$.

	As in the proof of \autoref{lem:resolution} we have $Q(\pi_{-}(\ca{O}_{\widehat{Z}}(-2,-1))) \simeq 0 \simeq Q(\pi_{-}(\ca{O}_{\widehat{Z}}(-1,-1)))$.
	Let us show the second statement for completeness;
	by \eqref{resol-P} and \eqref{resol-Z} we have $\ker(\ca{O}_{\widehat{Z}} \xrightarrow{u} \ca{O}_{\widehat{Z}}(-1,-1) ) = \ca{O}_{\mb{P}\times\mb{P}}(-1,-1) = \mathrm{coker}(\ca{O}_{\widehat{Z}} \xrightarrow{u} \ca{O}_{\widehat{Z}}(-1,-1) )$ and thus the statement follows.
	Hence, we get the distinguished triangle
	\begin{equation}
		\label{eqn:dt-psi-2-second}
		Q(\ca{O}_{\mb{P}\times\mb{P}}(0,-2))[1] \rightarrow \overline{\Psi}^{L}_{+}(\{ \ca{O}_{Z_{+}}(-1)^{\oplus 3} \xrightarrow{\mathbf{s}} \ca{O}_{Z_{+}}(-2)\}) \rightarrow Q(\ca{O}_{\mb{P}\times\mb{P}}(0,-2))
	\end{equation}
	An explicit computation using the adjunction $\overline{\Psi}^L_{+} \dashv \overline{\Psi}_{+}$ shows $\End^{\bullet}_{\derived_{\qc}(Z_{+}) \left/ \mcr{K} \right.}(\overline{\Psi}^{L}_{+}(E_2)) \simeq k[\varepsilon_2] / \varepsilon_2^2$, $\deg(\varepsilon_2) = -1$, $\Hom^{\bullet}_{\derived_{\qc}(Z_{+}) \left/ \mcr{K} \right.}(\overline{\Psi}^{L}_{+}(E_1), \overline{\Psi}^{L}_{+}(E_2)) = 0$, and
	\[
		\Hom^{\bullet}_{\derived_{\qc}(Z_{+}) \left/ \mcr{K} \right.}(\overline{\Psi}^{L}_{+}(E_2), \overline{\Psi}^{L}_{+}(E_1))
		\simeq
		\Hom^{\bullet}_{\derived_{\qc}(Z_{+})}(I_2', I_1').
	\]
	This proves the SOD because the dg-algebras $k[\varepsilon_i]/\varepsilon_i^2$ are intrinsically formal, and the statement for the gluing bimodules at the level of underlying graded modules.
	To get the statement at the level of dg-modules one can proceed analogously to the proof of \autoref{thm:C=D(R)}.

	The statement about the images of $k[\varepsilon_j]/\varepsilon_j^2$ follows by mapping via $\overline{\Psi}_{+}$ the distinguished triangles \eqref{eqn:dt-psi-1} and \eqref{eqn:dt-psi-2-second}.
\end{proof}

\begin{rmk}[Where boundedness fails, II]
	\label{rmk:where-boundedness-fails-II}
	We now present a second way in which boundedness fails: we show that the SOD $\derived_{\qc}(\widehat{Z}) = \gen{\mcr{K}, \mcr{S}_{+}, p^{\ast}\derived_{\qc}(Z_{-})}$ does not induce and SOD of $\derived^{\bounded}(\widehat{Z})$.
	Consider the case $n=2$ for simplicity.
	Then, joining \eqref{resol-P} and \eqref{resol-Z} we get the long exact sequence
	\[
		\dots \xrightarrow{w} \ca{O}_{\widehat{Z}}(1,-1) \xrightarrow{u} \ca{O}_{\widehat{Z}}(0,-2) \xrightarrow{w} \ca{O}_{\widehat{Z}}(1,-1) \xrightarrow{u} \ca{O}_{\widehat{Z}}(0,-2) \rightarrow \ca{O}_{\mb{P}\times\mb{P}}(0,-2).
	\]
	Using \eqref{eqn:psi-2} we see that we have a distinguished triangle
	\[
		\{ \dots \xrightarrow{\mathbf{y}} \ca{O}_{\widehat{Z}}(0,-1)^{\oplus 3} \xrightarrow{u\mathbf{x}} \ca{O}_{\widehat{Z}}(0,-2)\} \rightarrow \ca{O}_{\mb{P}\times\mb{P}}(0,-2) \rightarrow \bigoplus_{n \geq 1} \{ \ca{O}_{\widehat{Z}}(-2,-1) \xrightarrow{\mathbf{x}} \ca{O}_{\widehat{Z}}(-1,-1)^{\oplus 3}\}[2n]
	\]
	and the image of the above triangle via $\pi_{-}$ gives the decomposition of $\ca{O}_{\mb{P}\times\mb{P}}(0,-2)$ with respect to $\derived_{\qc}(\widehat{Z}) = \gen{\mcr{K}, \mcr{S}_{+}, p^{\ast}\derived_{\qc}(Z_{-})}$.
	Hence, we see that the projection of $\ca{O}_{\mb{P}\times\mb{P}}(0,-2)$ to $\mcr{K}$ is unbounded.
	Notice that this does not rule out the possibility that $\mcr{K}^b$ is left admissible in $\derived^{\bounded}(\widehat{Z})$.
\end{rmk}

\begin{thm}
	\label{thm:end-algebra-structure-sheaves-mukai}
	We have
	\[
		\Hom^{\bullet}_{\derived_{\qc}(\widehat{Z})\left/\mcr{K}\right.}(Q(\ca{O}_{\mb{P}\times\mb{P}}(0,-j)), Q(\ca{O}_{\mb{P}\times\mb{P}}(0,-i))) \simeq
		\left\{
		\begin{array}{lr}
			k[q_i] & i = j\\
			\Hom^{\bullet}_{\derived_{\qc}(Z_{+})}(\ca{O}_{\mb{P}}(-j),\ca{O}_{\mb{P}}(-i)) & j > i\\
			0 & i < j
		\end{array}
		\right.
	\]
	where $\deg(q_i) = 2$.
	Hence,
	\[
		\tgen{\oplus_{j=1}^n Q(\ca{O}_{\mb{P}\times\mb{P}}(0,-j))} \simeq \gen{\derived(k[q_n])^c, \dots, \derived(k[q_1])^c}
	\]
	where the right gluing bimodule between the $j$-th and the $i$-th component (counting right to left) is given by $\mathrm{RHom}_{X_{+}}(\ca{O}_{\mb{P}}(-j), \ca{O}_{\mb{P}}(-i))$.
	Furthermore, the functor $\overline{\Psi}_{+}$ restricts to a spherical functor
	\begin{equation}
		\label{eqn:psi-restricted-to-thick}
		\overline{\Psi}_{+} \vert_{\tgen{\oplus_{j=1}^n Q(\ca{O}_{\mb{P}\times\mb{P}}(0,-j))}} \colon \tgen{\oplus_{j=1}^n Q(\ca{O}_{\mb{P}\times\mb{P}}(0,-j))} \rightarrow \derived^{\bounded}(Z_{+}).
	\end{equation}

\end{thm}

\begin{proof}
	We prove that stament for $n=2$, the general case being analogous, and we use the same notation as in the proof of \autoref{thm:C=D(R)-Mukai-complete}.

	Notice that $\overline{\Psi}^L_{+}$ and $\overline{\Psi}^R_{+}$ preserve compactness because $\overline{\Psi}_{+}$ is spherical.
	By \eqref{eqn:dt-psi-1} and \eqref{eqn:dt-psi-2-second} we see that $\overline{E} \in \tgen{\oplus_{j=1}^2 Q(\ca{O}_{\mb{P}\times\mb{P}}(0,-j))}$.
	Therefore, as $\overline{E}$ is a compact generator of $\ker \overline{p}_{\ast}$, we get $(\ker \overline{p}_{\ast})^c \subset \tgen{\oplus_{j=1}^2 Q(\ca{O}_{\mb{P}\times\mb{P}}(0,-j))}$.
	Thus, the adjunctions $\overline{\Psi}^L_{+} \dashv \overline{\Psi}_{+} \dashv \overline{\Psi}^R_{+}$ restrict to adjunctions between $\tgen{\oplus_{j=1}^2 Q(\ca{O}_{\mb{P}\times\mb{P}}(0,-j))}$ and $\derived^{\bounded}(Z_{+})$.
	It follows that \eqref{eqn:psi-restricted-to-thick} is spherical because $\overline{\Psi}_{+}$ was.

	Notice that if we prove the statement about the graded homs, the claim about the SOD follows and the one about the gluing bimodules can be proved with the same arguments as in the proof of \autoref{thm:C=D(R)}.
	By \eqref{eqn:dt-psi-1} and \eqref{eqn:dt-psi-2-second} we get $\mathrm{hocolim} \{ X_n^j\} \simeq Q(\ca{O}_{\mb{P}\times\mb{P}}(0,-j))$ where
	\[
		X_{n}^j = \{ \overline{\Psi}^{L}_{+} (E_j)[n] \rightarrow \overline{\Psi}^{L}_{+} (E_j)[n-1] \rightarrow \dots \rightarrow \overline{\Psi}^{L}_{+} (E_j)[1] \rightarrow \overline{\Psi}^{L}_{+} (E_j)\}
	\]
	(see e.g. \cite[\href{https://stacks.math.columbia.edu/tag/0A5K}{Tag 0A5K}]{stacks-project} for the definition of homotopy colimit).
	The statement follows using the adjunction $\overline{\Psi}^{L}_{+} \dashv \overline{\Psi}_{+}$ because
	\[
		\Hom^{\bullet}_{\derived_{\qc}(Z_{+})}(E_j, \ca{O}_{\mb{P}}(-i)) \simeq
		\left\{
		\begin{array}{lr}
			k & i = j\\
			\Hom^{\bullet}_{\derived_{\qc}(Z_{+})}(I_j', \ca{O}_{\mb{P}}(-i)) & j > i\\
			0 & i < j
		\end{array}
		\right.
	\]
	and we have $\mathrm{hocolim} \{ I_j'[n] \rightarrow \dots \rightarrow I_j\} \simeq \ca{O}_{\mb{P}}(-j)$.
\end{proof}

\subsection{Other examples}
Let us walk through a few more examples where we can apply the theory of \autoref{sect:flop-flop-inverse} but where an explicit description of the category $\ker \overline{p}_{\ast}$ eludes our understanding.

\subsubsection{Grassmannian flops}

Let $V$ and $S$ be two vector spaces of dimension $n$ and $r$, respectively, with $r < n$ and consider the stack quotient
\[
	\mathfrak{X}^{r,n} = \left[ \Hom(S,V) \oplus \Hom(V,S) \left/ \mathrm{GL}(S) \right. \right]
\]
where the action is given by $M \cdot (a,b) = (a M^{-1}, Mb)$.
The GIT quotient associated to the linearisations $\ca{O}(1) := \det \, S^{\vee}$ and $\ca{O}(-1)$ are, respectively,
\[
	\mathfrak{X}^{ss}_{+} \left/ \mathrm{GL}(S) \right. = \{ (a,b) : b \text{ is surjective} \} \left/ \mathrm{GL}(S) \right. = \Tot \left( \Hom(S,V) \rightarrow \mathrm{Gr}(V,r) \right) =: X_{+},
\]
which is the total space of a vector bundle over the grassmannian of $r$ dimensional quotients in $V$, and
\[
	\mathfrak{X}^{ss}_{-} \left/ \mathrm{GL}(S) \right. = \{ (a,b) : a \text{ is surjective} \} \left/ \mathrm{GL}(S) \right. = \Tot \left( \Hom(V,S) \rightarrow \mathrm{Gr}(r,V) \right) =: X_{-},
\]
where $\mathrm{Gr}(r,V)$ is the grassmannian of $r$ dimensional subspaces of $V$.

The two varieties $X_{-}$ and $X_{+}$ are birational and derived equivalent, see \cite{Donovan-Segal-Grassmannian-twists}, \cite{Halpern-GIT-15}.
The fibre product over the common singularity is given by
\[
	\widehat{X} = \Tot \left( \Hom(Q,S) \rightarrow \mathrm{Gr}(r,V) \times \mathrm{Gr}(V,r) \right)
\]
where $Q$ is the tautological quotient bundle and $S$ is the tautological subbundle.

In \cite{Ballard-Chidambaram-Nitin-Favero-McFaddin-Vandermolen-Kernel-Gr-flop} it was proved that $\widehat{X}$ gives a derived equivalence between $\derived^{\bounded}(X_{-})$ and $\derived^{\bounded}(X_{+})$.
As everything is smooth, by \cite[Lemma 2.12]{KL15} we obtain an equivalence $\derived_{\qc}(X_{-}) \simeq \derived_{\qc}(X_{+})$, and we can apply \autoref{thm:spherical-functor-from-sods}.
In \cite{Ballard-Chidambaram-Nitin-Favero-McFaddin-Vandermolen-Kernel-Gr-flop} it is also (implicitly) proved that the flop-flop functor is given by a composition of $n-r$ window shifts, see [{\it ibidem}, Corollary 5.2.9, Corollary 5.2.10].
Moreover, by \cite{Donovan-Segal-Grassmannian-twists} we know that every window shift is realised as the spherical twist around a spherical functor whose source category is given by $\derived_{\qc}(\mathfrak{X}^{r-1,n})$.
Therefore, one would expect $\ker \overline{p}_{\ast}$ to have a SOD reflecting this splitting.
However, the picture is more complicated.
%The reason, as it explained in the example of the Abuaf flop, is that the glued spherical functor is not conservative.

\subsubsection{Abuaf flop}

Consider $V$ a symplectic vector space of dimension $4$ and let $\mb{P} = \mb{P}V$, $\mathrm{LGr} = \mathrm{LGr}(2,V)$.
Then, consider
\[
	\begin{array}{lcr}
		X_{-} = \Tot (L^{\perp} \left/ L \right. \otimes L^2 \rightarrow \mb{P}) & \mathrm{and} & \quad X_{+} = \Tot ( S(-1) \rightarrow \mathrm{LGr})
	\end{array}
\]
where $L$ is the tautological subbundle on $\mb{P}$, $S$ is the tautological subbundle on $\mathrm{LGr}$, and $\ca{O}_{\mathrm{LGr}}(-1) = \BigWedge^2 S$.
The varieties $X_{-}$ and $X_{+}$ are birational and derived equivalent \cite{Seg-Abuaf-Flop}; furthermore, they resolve the same singularity $Y = \mathrm{Spec} \, H^{0}(X_{\pm}, \ca{O}_{X_{\pm}})$.
In \cite{Hara17} the author proves that the structure sheaf of the fibre product $X_{-} \times_Y X_{+}$ gives an equivalence from $\derived^{\bounded}(X_{-})$ to $\derived^{\bounded}(X_{+})$, and thus it gives an equivalence in both directions \cite[Theorem 1.1]{Bridgeland-equivalences-triangulated-categories}.

It turns out that there are two families of tilting bundles on both sides of the flop and that one of them gives the flop functor in one direction, \cite[Thereom 4.5]{Hara17}.
It can be proved that the other family gives the flop functor in the other direction.
Working out the exact numbers, we get an isomorphism of functors
\begin{equation}
	\label{eqn:flop-flop-abuaf}
	q_{\ast} p^{\ast} p_{\ast} q^{\ast} \simeq T^{-1}_{j_{\ast}S}  T^{-1}_{j_{\ast}\ca{O}_{\textup{LGr}}(-1)}  T^{-1}_{j_{\ast}S(-1)}  T^{-1}_{j_{\ast}\ca{O}_{\textup{LGr}}(-2)}  T^{-1}_{j_{\ast}S(-2)}
\end{equation}
where $j : \mathrm{LGr} \hookrightarrow X_{+}$ is the inclusion of the zero section.
Given this decomposition of the flop-flop functor, one might expect the category $\ker \overline{p}_{\ast}$ to have a full, exceptional collection of length five.
However, there is a fundamental difference between this example and standard flops: the objects $j_{\ast}S$, $i_{\ast} \ca{O}_{\mathrm{LGr}}(-1)$, $j_{\ast}S(-1)$, $i_{\ast}\ca{O}_{\mathrm{LGr}}(-2)$, and $j_{\ast}S(-2)$ are not \say{independent} in the derived category: we have the following short exact sequence on $\mathrm{LGr}$
\begin{equation}
	\label{eqn:ex-seq-LGr}
	S(-1) \rightarrow V^{\ast} \otimes \ca{O}_{\mathrm{LGr}}(-1) \rightarrow S.
\end{equation}
Hence, the functor obtained by gluing the twists in \eqref{eqn:flop-flop-abuaf} using the construction of \cite{Barb-Spherical-twists} is not conservative, while the spherical functor $\overline{\Psi}_{+}$ is.
We can guess what $\ker \overline{p}_{\ast}$ should look like as follows;
the fibre product $\widehat{X}$ is the gluing of $\Bl_{\mathrm{LGr}}X_{+}$ and $\mb{P} \times \mathrm{LGr}$ along $\mb{P}(S(-1))$, which is embedded in the latter via the exact sequence \eqref{eqn:ex-seq-LGr}.
We can prove that $\ker \overline{p}_{\ast}$ is generated by the objects $i_{\ast} \ca{O}_{\mathrm{LGr}}(-1)$, $i_{\ast}\ca{O}_{\mathrm{LGr}}(-2)$, $i_{\ast} S(-1)$ where $i : \mb{P} \times \mathrm{LGr} \hookrightarrow \widehat{X}$ is the closed embedding, and therefore one might conjecture that $\ker \overline{p}_{\ast}$ is the quotient of the source category obtained by gluing the spherical twists in \eqref{eqn:flop-flop-abuaf} by the kernel of the associated spherical functor.
However, at the moment we do not know how to prove whether this is right or wrong.

%APPENDIX
\appendix
\section{Computations}

\subsection{Standard flops (local model)}
\label{sect:appendix-std-flops}
%%%APPENDIX STANDARD FLOPS

We use the notation of \autoref{sect:flop-flop-inverse} and \autoref{section:std-flops};
in particular, $\overline{\Psi}_{+} \colon \ker \overline{p}_{\ast} \rightarrow \dqc$ and its left adjoint is $\overline{\Psi}_{+}^L = Q  \pi_{-}  p^{\ast}$, $\pi_{-} = \cone(p^{\ast}p_* \rightarrow \id)$.

First of all, notice that on $X_{+}$ we have the long exact sequence
\begin{equation}
	\label{eqn:Kosz-X-+}
	\ca{O}_{X_{+}}(-n-1) \rightarrow \ca{O}_{X_{+}}(-n)^{\oplus n+1} \rightarrow \dots \rightarrow \ca{O}_{X_{+}}(-1)^{\oplus n+1} \rightarrow \ca{O}_{X_{+}}
\end{equation}
given by pulling up the Koszul complex that exists on $\mb{P}$, and that on $\widehat{X}$ we have the following short exact sequence
\begin{equation}
	\label{eqn:reg-section-std-flops}
	\ca{O}_{\widehat{X}}(1,1) \xrightarrow{u} \ca{O}_{\widehat{X}} \rightarrow \ca{O}_{\mb{P}\times\mb{P}}.
\end{equation}
Moreover, notice that we have
\begin{equation}
	\label{eqn:image-j-std-flops}
		\overline{\Psi}_{+}^L(\ca{O}_{X_{+}}(-j)) = Q \left( \left\{\ca{O}_{\widehat{X}}(j,0) \xrightarrow{u^j} \ca{O}_{\widehat{X}}(0,-j) \right\} \right) \simeq Q(\ca{O}_{\widehat{X}}(0,-j) \otimes \ca{O}_{\widehat{X}} \left/ \ca{I}_{\mb{P} \times \mb{P}}^{\, j} \right.)
\end{equation}

\begin{lem}
	\label{lem:resolution}
	For $j = 1, \dots, n$ we have an isomorphism
	\[
		\left\{ \overline{\Psi}_{+}^{L} (\ca{O}_{X_{+}}(-1))^{\oplus \binom{n+1}{j-1}} \rightarrow \overline{\Psi}_{+}^{L} (\ca{O}_{X_{+}}(-2))^{\oplus \binom{n+1}{j-2}} \rightarrow \dots \rightarrow \overline{\Psi}_{+}^{L} (\ca{O}_{X_{+}}(-j)) \right\} \simeq Q(\ca{O}_{\mb{P} \times \mb{P}}(0,-j))
	\]
	induced by the map $\overline{\Psi}_{+}^L \ca{O}_{X_{+}}(-j) \rightarrow Q(\ca{O}_{\mb{P} \times \mb{P}}(0,-j))$ that comes from adjunction.
\end{lem}

\begin{proof}
	We prove the statement for $n = 2$, the general case being analogous.
	Equation \eqref{eqn:image-j-std-flops} is the statement for $j=1$.
	For $j=2$ notice that we have the long exact sequence
	\begin{equation}
		\label{eqn:les-std-flop-proof}
		\ca{O}_{\widehat{X}}(-2, -1) \rightarrow \ca{O}_{\widehat{X}}(-1,-1)^{\oplus 3} \rightarrow \ca{O}_{\widehat{X}}(0, -1)^{\oplus 3} \rightarrow \ca{O}_{\widehat{X}}(0,-2) \rightarrow \ca{O}_{\mb{P}\times\mb{P}}(0,-2)
	\end{equation}
	obtained by pulling up \eqref{eqn:Kosz-X-+} to $\widehat{X}$, tensoring it with $\ca{O}_{\widehat{X}}(1,-1)$, and joining it with \eqref{eqn:reg-section-std-flops} tensored with $\ca{O}_{\widehat{X}}(0,-2)$.
	Tensoring the short exact sequence \eqref{eqn:reg-section-std-flops} with $\ca{O}_{\widehat{X}}(-1,-1)$ we obtain
	\[
		\ca{O}_{\widehat{X}} \xrightarrow{u} \ca{O}_{\widehat{X}}(-1,-1) \rightarrow \ca{O}_{\mb{P}\times\mb{P}}(-1,-1),
	\]
	and as $\ca{O}_{\mb{P}\times\mb{P}}(-1,-1) \in \mcr{K}$ we have $Q(\pi_{-}(\ca{O}_{\widehat{X}}(-1,-1))) \simeq Q(\pi_{-}(\ca{O}_{\widehat{X}})) \simeq 0$.
	Similarly, one can prove $Q(\pi_{-}(\ca{O}_{\widehat{X}}(-2,-1))) \simeq 0$.
	Therefore, applying $Q  \pi_{-}$ to \eqref{eqn:les-std-flop-proof} and using \eqref{eqn:image-j-std-flops}, we get
	\[
		Q(\ca{O}_{\mb{P}\times\mb{P}}(0,-2)) \simeq \left\{ \overline{\Psi}_{+}^L (\ca{O}_{X_{+}}(-1))^{\oplus 3} \rightarrow \overline{\Psi}_{+}^L(\ca{O}_{X_{+}}(-2))  \right\},
	\]
	which is the statement.
\end{proof}

\begin{lem}
	\label{lem:cpt-gen-std-flops}
	The object $\oplus_{j=1}^n Q(\ca{O}_{\mb{P}\times\mb{P}}(0,-j))$ is a compact generator of $\ker \overline{p}_{\ast}$.
\end{lem}

\begin{proof}
	Notice that $\oplus_{j=0}^n \ca{O}_{X_{+}}(-j)$ is a compact generator for $\dqc$, and thus, by \autoref{thm:4-sods-quotient}, we get that $\oplus_{j=0}^n \overline{\Psi}_{+}^L(\ca{O}_{X_{+}}(-j))$ is a compact generator for $\ker \overline{p}_{\ast}$.

	As $\overline{\Psi}_{+}^L (\ca{O}_{X_{+}}) = 0$, \autoref{lem:resolution} tells us that
	\[
		\gen{ \left\{ \overline{\Psi}_{+}^{L}(\ca{O}_{X_{+}}(-j)) \right\}_{j = 1, \dots, n} } = \gen{ \left\{ \ca{O}_{\mb{P} \times \mb{P}}(0,-j) \right\}_{j = 1, \dots, n} } .
	\]
	Hence, $\oplus_{j=1}^n Q(\ca{O}_{\mb{P}\times\mb{P}}(0,-j))$ is a generator.

	As $\overline{\Psi}_{+}$ is continuous, $\overline{\Psi}_{+}^L$ preserves compactness.
	By \autoref{lem:resolution} each $Q(\ca{O}_{\mb{P}\times\mb{P}}(0,-j))$ is the image via $\overline{\Psi}_{+}^L$ of a bounded complex of coherent sheaves on $X_{+}$.
	As $X_{+}$ is smooth, such a complex is compact, and thus we get that $\oplus_{j=1}^n Q(\ca{O}_{\mb{P}\times\mb{P}}(0,-j))$ is compact.
\end{proof}

\begin{lem}
	\label{lem:end-algebra-std-flops}
	The pushforward functor $\overline{q}_{\ast} : \dk \rightarrow \derived_{\qc}(X_{+})$ induces an isomorphism of graded algebras $\End_{\dk}^{\bullet}(\oplus_{j=1}^n Q(\ca{O}_{\mb{P}\times\mb{P}}(0,-j))) \simeq \Hom^{\rightarrow}_{\dqc}(\{\ca{O}_{\mb{P}}(-j)\})$.
\end{lem}

\begin{proof}
	This follows from \autoref{lem:resolution} together with the adjunction $\overline{\Psi}_{+}^L \dashv \overline{\Psi}_{+}$.
\end{proof}

\subsection{Standard flops (family case)}
\label{sect:appendix-std-flops-families}
We use the notation of \autoref{sect:flop-flop-inverse} and \autoref{section:std-flops-families}.

\begin{lem}
	\label{lem:K-in-SOD-appendix}
	In SOD \eqref{SOD-roof-families} we have $\mcr{K} = \mcr{A}$ and $\mcr{S}_{+} = \mcr{B}$ where $\mcr{S}_{+} = \cgen{\im(\pi_{-}  q^{\ast})}$, $\pi_{-} = \cone(p^{\ast}p_{\ast} \rightarrow \id )$.
\end{lem}

\begin{proof}
	It is enough to prove $\mcr{K} = \mcr{A}$ as the other statement will then follow from \autoref{thm:formal-statement-4-sods}.
	Given an object $K \in \mcr{K}$, its projection to $p^{\ast}\derived_{\qc}(X_{-})$ is zero.
	Therefore, $K$ can be decomposed in terms of the remaining subcategories in SOD \eqref{SOD-roof-families}.
	In particular, there exists a distinguished triangle $K' \rightarrow K \rightarrow K''$ where $K' \in \mcr{B}$ and $K'' \in \mcr{A}$.

	Our aim is to show that $K'$ is zero, so we can assume we are in the following local situation: $Z$ is affine, $\mb{P}V_{+} \simeq \mb{P}^n \times Z \subset X_{+}$.
	As $K' \in \mcr{B}$ we know that there exist maps
	\begin{equation}
		\label{eqn:filtration-k}
		0 = E_0 \rightarrow E_1 \rightarrow \dots \rightarrow E_{n} = K'
	\end{equation}
	where $E_i \in \mcr{B}$ and $\cone(E_i \rightarrow E_{i+1}) = \mb{R}_{-i} I_{0,-i-1}(F_{i+1})$, $F_{i+1} \in \derived_{\qc}(Z)$, $i=0,\dots,n-1$.

	An explicit computation shows that
	\[
		(\omega_{+})_{\ast} q_{\ast} (\mb{R}_{-i} I_{0,-i-1}(F_{i+1}) \otimes q^{\ast} N^{\vee}_{\mb{P}V_{+}/X_{+}}) \simeq 0
	\]
	for $i = 1, \dots, n$, and $(\omega_{+})_{\ast} q_{\ast}( I_{0,-1}(F_{1}) \otimes q^{\ast} N^{\vee}_{\mb{P}V_{+}/X_{+}}) \simeq \mathrm{R}\Gamma(F_1)^{\oplus n+1}$.
	As $q_{\ast}K' \simeq 0$, tensoring \eqref{eqn:filtration-k} with $q^{\ast} N^{\vee}_{\mb{P}V_{+}/X_{+}}$ and pushing down to $Z$ we get $\mathrm{R}\Gamma(F_1)^{\oplus n+1} \simeq 0$, and as $Z$ is affine, $F_1 \simeq 0$.

	Proceeding inductively using exterior powers of the dual of the normal bundle we deduce that $F_i \simeq 0$ for every $i$, and therefore $K' \simeq 0$.
\end{proof}

We will make use of the following notation: if $\ca{E}$ and $\ca{F}$ are two Fourier--Mukai kernels for functors $\derived_{\qc}(X_1) \rightarrow \derived_{\qc}(X_2)$, $\derived_{\qc}(X_2) \rightarrow \derived_{\qc}(X_3)$, the Fourier--Mukai kernel for the composition will be denoted as $\ca{F} \ca{E} := (p_{13})_{\ast} (p_{12}^{\ast}\ca{E} \otimes p_{23}^{\ast}\ca{F})$, with the standard notation for the projections $p_{ij} : X_1 \times X_2 \times X_3 \rightarrow X_{i} \times X_{j}$.

We will denote the Fourier--Mukai kernel for $\mb{R}_{a,b}$, $I_{a,b}$, $I_{a,b}^{L}$ respectively as $\ca{R}_{a,b}$, $\ca{I}_{a,b}$, and $\ca{I}_{a,b}^{L}$.
Explicitly, they are given by
\[
	\renewcommand{\arraystretch}{1.5}
	\begin{array}{l}
		\ca{I}_{a,b} = (\omega_{+}  \alpha_{+} \times i)_{\ast} \Delta_{\ast} \ca{O}_{\mb{P}V_{-} \times_{Z} \mb{P}V_{+}}(a,b),\\
		\ca{I}_{a,b}^{L} = \sigma^{\ast}\ca{I}_{a,b}^{\vee} \otimes p^{\ast}_{\widehat{X}} \omega_{\widehat{X}} [\dim \widehat{X}],\\
		\ca{R}_{a,b} = \cone \left( \Delta_{\ast} \ca{O}_{\widehat{X}} \rightarrow \ca{I}_{a,b} \ca{I}_{a,b}^{L} \right)[-1],
	\end{array}
\]
where, in the following, $\sigma : \widehat{X} \times Z \rightarrow Z \times \widehat{X}$ and $\Delta$ will denote the twist morphism and the diagonal inclusion, respectively.
The map in the third line is given by adjunction.

Let us denote $\ca{R}_{-j}$ the kernel for the functor $\mb{R}_{-j}$ obtained by convolving the kernels of the various $\mb{R}_{a,b}$, and $\ca{R}_{-j+1} \ca{I}_{0,-j}$ the kernel for $\mb{R}_{-j+1}  I_{0,-j}$.

Finally, we denote the kernel of the functor $J_{-i}(-) = (j_{+})_{\ast}(\ca{O}_{\mb{P}V_{+}}(-i) \otimes \omega_{+}^{\ast}-)$ as $\ca{J}_{-i}  = (\omega_{+} \times j_{+})_{\ast} \Delta_{\ast} \ca{O}_{\mb{P}V_{+}}(-i)$.

\begin{lem}
	\label{lem:pushdown-non-family}
	We have an isomorphism $(\id \times q)_{\ast} \ca{R}_{-i+1} \ca{I}_{0,-i} \simeq \ca{J}_{-i}$ for $1 \leq i \leq n$.
\end{lem}

\begin{proof}
	By construction $\ca{R}_{a,b}$ has a map to $\Delta_{\ast} \ca{O}_{\widehat{X}}$.
	Using these maps we get a map $\eta: \ca{R}_{-i+1} \ca{I}_{0,-i} \rightarrow \ca{I}_{0,-i}$.

	It is easy to see that we have an isomorphism $(\id \times q)_{\ast}\ca{I}_{0,-i} \simeq \ca{J}_{-i}$.
	We claim that $(\id \times q)_{\ast} \cone(\eta) \simeq 0$.
	This follows from the fact that $\cone(\ca{R}_{a,b} \rightarrow \Delta_{\ast}\ca{O}_{\widehat{X}}) \simeq \ca{I}_{a,b}\ca{I}_{a,b}^L$ and that we are taking $a,b \in \{-n, \dots, -1\}$.
\end{proof}

\begin{prop}
	\label{lem:natural-transf-family-case-appendix}
	For $1 \leq j < i \leq n$ we have an isomorphism 
	\[
		\mathrm{RHom}_{Z \times \widehat{X}}(\ca{R}_{-i+1} \ca{I}_{0,-i}, \ca{R}_{-j+1} \ca{I}_{0,-j}) \xrightarrow{(\id \times q)_{\ast}} \mathrm{RHom}_{Z \times X_{+}}(\ca{J}_{-i}, \ca{J}_{-j}).
	\]
\end{prop}

\begin{proof}
	By the above lemma, we get a map
	\begin{equation}
	\label{eqn:nat-transf-kernels}
			(\id\times q)_{\ast} \mathrm{R} \ca{H}om_{Z \times \widehat{X}}(\ca{R}_{-i+1}\ca{I}_{0,-i}, \ca{R}_{-j+1}\ca{I}_{0,-j}) \rightarrow \mathrm{R} \ca{H}om_{Z \times X_{+}}(\ca{J}_{-i}, \ca{J}_{-j}).
	\end{equation}

	What we want to do now is to find a deformation over $\mb{A}^1$ of this map whose fibre over $0$ gives the local model of \autoref{section:std-flops}.
	Then, we will obtain the isomorphism we want from this.

	Recall the deformation to the normal bundle, \cite[Chapter 5]{Fulton-intersection-theory}.
	Given any closed immersion $X \hookrightarrow Y$, we can construct a flat family $f : \ca{X} \rightarrow \mb{A}^1$, together with a closed immersion $X \times \mb{A}^1 \hookrightarrow \ca{X}$, such that for any $t \neq 0$ we have $\ca{X}_{t} = Y$ with the given embedding, and $\ca{X}_{0} = \Tot(N_{X/Y})$ with the embedding given by the zero section.
	Let us denote $\widehat{f} : \widehat{\ca{X}} \rightarrow \mb{A}^1$ the deformation to the normal bundle of $\widehat{X}$, and similarly $f_{\pm} : \ca{X}_{\pm} \rightarrow \mb{A}^1$ the deformation to the normal bundle of $X_{\pm}$.
	The maps between $\widehat{\ca{X}}$ and $\ca{X}_{\pm}$ that lift $p$ and $q$ will still be denoted $p$ and $q$.
	We will employ the following notation
	\[
		\begin{tikzcd}
			\widehat{\ca{X}_t} \ar[r, "r_t", hook] & \widehat{\ca{X}} \\
			\mb{P}V_{-} \times_{Z} \mb{P}V_{+} \times \{ t \} \ar[u, hook, "i_{t}"] \ar[r, hook, "s_{t}"] & \mb{P}V_{-} \times_Z \mb{P}V_{+} \times \mb{A}^1 \ar[u, hook, "s"]
		\end{tikzcd}
		\begin{tikzcd}[row sep = 0.5em]
			\mb{P}V_{+} \times \mb{A}^1 \ar[r, "j_{+}", hook] & \ca{X}_{+},\\
			Z \times \{t \} \ar[r, "u_{t}", hook] & Z \times \mb{A}^1.
		\end{tikzcd}
	\]

	We want to prove that the functors $I_{a,b}$ and $\mb{R}_{-i}$ can be constructed flatly in the family.
	As all the varities appearing are flat over $\mb{A}^1$, it means that the kernels representing $I_{a,b}^{\text{fam}}$ and $\mb{R}^{\text{fam}}_{-i}$ must be objects $\ca{I}_{a,b}^{\text{fam}} \in \derived^{\bounded}(Z \times \widehat{\ca{X}})$, $\ca{R}^{\text{fam}}_{-i} \in \derived_{\qc}(\widehat{\ca{X}} \times_{\mb{A}^1} \widehat{\ca{X}})^c$ (they have to be perfect complexes as the functors they induce preserve boundedness and coherence).

	Let us denote $\xi : Z \times \mb{A}^1 \rightarrow Z$ the projection.
	Then, we define
	\[
		\ca{I}_{a,b}^{\text{fam}} := (\xi \times \id)_{\ast} (\id \times s)_{\ast} (\omega_{+}  \alpha_{+} \times \id \times \id)_{\ast} \Delta_{\ast} \ca{O}_{\mb{P}V_{-} \times \mb{P}V_{+} \times \mb{A}^1}(a,b).
	\]
	The Fourier--Mukai transform induced by this complex is the sought family version of $I_{a,b}$.
	Indeed, a long (but simple) exercise in base change for Tor independent diagram \cite[\href{https://stacks.math.columbia.edu/tag/08IB}{Tag 08IB}]{stacks-project} shows
	\[
		(\id \times r_t)^{\ast} \ca{I}_{a,b}^{\text{fam}} \simeq (\omega_{+}  \alpha_{+} \times i_t)_{\ast} \Delta_{\ast} \ca{O}_{\mb{P}V_{-} \times_Z \mb{P}V_{+} \times \{t\}}.
	\]
	Then
	\[
		\left( \ca{I}_{a,b}^{\text{fam}} \right)^{L} := \sigma^{\ast}\left( \ca{I}_{a,b}^{\text{fam}} \right)^{\vee} \otimes p_{\widehat{\ca{X}}}^{\ast}\omega_{\widehat{\ca{X}}}[\dim \widehat{\ca{X}} -1]
	\]
	gives the kernel for the left adjoint of $\ca{I}_{a,b}^{\mathrm{fam}}$.
	Notice that we have
	\[
		(r_t \times \id)^{\ast} \left( \ca{I}_{a,b}^{\text{fam}} \right)^{L} \simeq \sigma^{\ast}\left( (\id \times r_t)^{\ast} \ca{I}_{a,b}^{\text{fam}} \right)^{\vee} \otimes p_{\widehat{\ca{X}}_t}^{\ast} \omega_{\widehat{\ca{X}}_t} [\dim \widehat{\ca{X}}_t],
	\]
	and therefore $\left( \ca{I}_{a,b}^{\text{fam}} \right)^{L}$ restricts to the Fourier--Mukai kernel for $I_{a,b}^L$ on every fibre.

	We define $\ca{R}_{a,b}^{\text{fam}}$ as the cone of the map $\Delta_{\ast} \ca{O}_{\widehat{\ca{X}}} \rightarrow \ca{I}_{a,b}^{\text{fam}}\left( \ca{I}_{a,b}^{\text{fam}} \right)^{L}$ which is constructed by adjunction.
	As the intersection of $\ca{X}_{t} \times \ca{X}_{t}$ and $\left( \widehat{\ca{X}} \times Z \right) \times_{\mb{A}^1} \widehat{\ca{X}}$ in $\widehat{\ca{X}} \times_{\mb{A}^1} \widehat{\ca{X}}$ is Tor independent, the restriction of $\ca{R}_{a,b}^{\text{fam}}$ to every fibre gives the kernel for $\mb{R}_{a,b}$.

	Finally, we define $\ca{R}_{-i}^{\text{fam}}$ as the convolution of the kernels $\ca{R}_{a,b}^{\text{fam}}$ to complete our construction of the family versions of the kernels.

	Define the functor $J_{-i}^{\text{fam}}$ as the one associated to the kernel
	\[
		\ca{J}^{\text{fam}}_{-i} = (\xi \times \id)_{\ast}(\omega_{+} \times \id \times j_{+})_{\ast}\Delta_{\ast} \ca{O}_{\mb{P}V_{-} \times \mb{A}^1}(-i) \in \derived^{\bounded}(Z \times \ca{X}_{+}).
	\]

	Even for the family version, we have $(\id \times q)_{\ast} \ca{R}^{\text{fam}}_{-i+1}\ca{I}^{\text{fam}}_{0,-i} \simeq \ca{J}^{\text{fam}}_{-i}$.
	The argument is the same as in \autoref{lem:pushdown-non-family}: we construct a morphism $\eta: \ca{R}^{\text{fam}}_{-i+1}\ca{I}^{\text{fam}}_{0,-i} \rightarrow \ca{I}^{\text{fam}}_{0,-i}$ and to prove that it is an isomorphism we restrict to each fibre and use the Tor independence of the following diagram
	\begin{equation}
		\label{eqn:tor-independence-diagram-std-flops-families}
		\begin{tikzcd}
			Z \times \widehat{\ca{X}}_t \ar[r, "\id \times r_t", hook] \ar[d, "\id \times q"'] & Z \times \widehat{\ca{X}} \ar[d, "\id \times q"]\\
			Z \times \left( \ca{X}_{+} \right)_{t} \ar[r, hook] & Z \times \ca{X}_{+}
		\end{tikzcd}.
	\end{equation}

	Therefore, map \eqref{eqn:nat-transf-kernels} lifts to a map
	\begin{equation}
	\label{eqn:nat-trans-family}
			(\id \times q)_{\ast} \mathrm{R} \ca{H} om_{\derived^{\bounded}(Z \times \widehat{\ca{X}})}(\ca{R}^{\text{fam}}_{-i+1}\ca{I}^{\text{fam}}_{0,-i}, \ca{R}^{\text{fam}}_{-j+1}\ca{I}^{\text{fam}}_{0,-j}) \xrightarrow{\varphi} \mathrm{R} \ca{H} om_{\derived^{\bounded}(Z \times \ca{X}_{+})}(\ca{J}^{\text{fam}}_{-i}, \ca{J}^{\text{fam}}_{-j}).	
	\end{equation}
	As the family is trivial outside $t = 0$, the restrictions of the above map to the fibres $(\ca{X}_{+})_{t}$, $t \neq 0$, give the map \eqref{eqn:nat-transf-kernels}.
	Therefore, if we call $S = \cone(\varphi)$, it is enough to prove $\mathrm{R}\Gamma(S_{t}) = 0$ for $0 \neq t \in \mb{A}^1$.

	Assume that $\mathrm{R}\Gamma(S_{0}) = 0$, the we claim that the result follows.
	Indeed, by the flatness of $f_{+}$ and base change \cite[\href{https://stacks.math.columbia.edu/tag/08IB}{Tag 08IB}]{stacks-project} we have $\mathrm{R}\Gamma(S_0) \simeq ((f_{+})_{\ast}S)_{0}$.
	Hence, as $\mathrm{Supp}((f_{+})_{\ast}S)$ is closed, if $0 \neq \mathrm{Supp}((f_{+})_{\ast}S)$ then there exists $t \neq 0$ not in the support of $(f_{+})_{\ast}S$, and using again the flatness of $f_{+}$ we get $0 \simeq ((f_{+})_{\ast}S)_{t} \simeq \mathrm{R}\Gamma(S_{t})$, which proves the claim.

	Therefore, to conclude we have only to prove that $\mathrm{R}\Gamma(S_{0}) = 0$; this is done in the following lemma.
\end{proof}

\begin{lem}
	\label{lem:fibre-over-0}
	Let $S$ be the cone of \eqref{eqn:nat-trans-family}, then $\mathrm{R}\Gamma(S_{0}) = 0$.
\end{lem}

\begin{proof}
	To prove this lemma, we will prove that on the fibre $(\ca{X}_{+})_0$ the map becomes the map in the (semi)local model.
	This means that it becomes the map for the setting \eqref{eqn:standard-flop-families} with $X_{+} = \Tot\left( \ca{O}_{\mb{P}V_{+}}(-1)^{\oplus n+1}\right)$.
	This will suffice, as we know that the statement of the lemma is true in the local case.
	To be precise, in \autoref{sect:appendix-std-flops} we proved the statement (1) for the case $Z = \text{pt}$, (2) for the complexes $\mb{R}_{-i+1}\ca{O}_{\mb{P}\times \mb{P}}(0,-i)$.
	In order, we solve these problems as follows.
	\begin{enumerate}[(1)]
		\item Proving that the derived global sections vanish is a local question in $Z$.
		Hence, we can assume that $Z$ is affine and small enough such that $V_{+}$ is trivial.
		If this is the case, then $\mb{P}V_+ \simeq Z \times \mb{P}$, and \eqref{eqn:standard-flop-families} becomes
		\[
			\begin{tikzcd}
				Z \times \mb{P} \ar[d, "p_1"] \ar[r, "{\id \times j_{+}}"] &[3em] Z \times \Tot\left( \ca{O}_{\mb{P}}(-1)^{\oplus n+1}\right)\\
				Z. & {} 
			\end{tikzcd}
		\]
		The functor $J_{-i}$ is now given by $J_{-i}(F) \simeq p_1^{\ast} F \otimes p_2^{\ast}(j_{+})_{\ast}\ca{O}_{\mb{P}}(-i)$.
		Hence, its kernel is
		\[
			(p_1 \times \id \times j_{+})_{\ast} \Delta_{\ast} p_2^{\ast} \ca{O}_{\mb{P}}(-i) \simeq (\Delta \times j_{+})_{\ast} (\ca{O}_{Z \times \mb{P}}(-i)) \simeq  \Delta_{\ast} \ca{O}_Z \boxtimes (j_{+})_{\ast}\ca{O}_{\mb{P}}(-i).
		\]
		Similarly, we have
		\[
			\begin{array}{lcr}
				\ca{I}_{a,b} \simeq \Delta_{\ast} \ca{O}_Z \boxtimes \ca{O}_{\mb{P} \times \mb{P}}(a,b) & \mathrm{and} & \ca{R}_{-j+1} \ca{I}_{0,-j} \simeq \Delta_{\ast} \ca{O}_Z \boxtimes \mb{R}_{-j+1} \ca{O}_{\mb{P} \times \mb{P}}(0,-j).
			\end{array}
		\]
		As all the derived homs we are interested in are between perfect complexes, these splittings as an exterior tensor product tell us that the derived homs split in a part coming from $Z \times Z$, and a part coming from either $\Tot(\ca{O}_{\mb{P} \times \mb{P}}(-1,-1))$ or $\Tot(\ca{O}_{\mb{P}}(-1)^{\oplus n+1})$.
		Hence, it is enough to prove that statement for the case $Z = \text{pt}$.
		\item When $Z = \text{pt}$, the functors we are interested in have source category equal to $\derived(k)$, hence their Fourier--Mukai kernels are given by the objects to which they map $k$.
		More precisely, in this case they are given by $\mb{R}_{-j+1} \ca{O}_{\mb{P} \times \mb{P}}(0,-j)$.
	\end{enumerate}
	
	That the map restricts to the local one on the fibre $(\ca{X}_{+})_0$ follows from the Tor independence of \eqref{eqn:tor-independence-diagram-std-flops-families}.
\end{proof}

%\subsection{Mukai flops}
%\label{sect:appendix-Mukai-flops}
%\input{Work-in-progress/Appendix-B}

%% BIBLIOGRAPHY

%\bibliography{/Users/Federico/Documents/PhD/Bibliography/Bibliography}
\bibliographystyle{alphaurl}

\newcommand{\etalchar}[1]{$^{#1}$}

\end{document}